\documentclass[12pt,reqno]{amsart}
\usepackage{enumerate}
\usepackage{amsrefs}
\usepackage{amsfonts, amsmath, amssymb, amscd, amsthm, bm, cancel}
\usepackage{url}
\usepackage{graphicx}
\usepackage[
linktocpage=true,colorlinks,citecolor=magenta,linkcolor=blue,urlcolor=magenta]{hyperref}
\usepackage{multicol}
\usepackage{comment}
\usepackage[margin=1in]{geometry}
\newtheorem{thm}{Theorem}[section]

  \newtheorem{lem}{Lemma}[section]

 \theoremstyle{definition}
 
  \newtheorem*{ack}{Acknowledgments}
 \theoremstyle{remark}
\newtheorem{rem}{Remark}[section]

 \numberwithin{equation}{section}

\allowdisplaybreaks

\newcommand{\f}{\left(}
\renewcommand{\r}{\right)}

\newcommand{\R}{\mathbb{R}}

\renewcommand{\a}{\alpha}
\renewcommand{\b}{\beta}

\renewcommand{\d}{\delta}
\newcommand{\e}{\varepsilon}
\renewcommand{\k}{\kappa}

\renewcommand{\t}{\theta}
\newcommand{\s}{\sigma}

\newcommand{\metric}[2]{\ensuremath{\langle #1, #2\rangle}}  

\begin{document}
\title{$C^2$ estimates for $k$-Hessian equations and a rigidity theorem}

\author{RUIJIA ZHANG}
\address{Key Laboratory of Pure and Applied Mathematics School of Mathematical Sciences, Peking University, Beijing 100871, P. R. China}
\email{\href{mailto:zhangrj17@mails.tsinghua.edu.cn}{zhangrj@pku.edu.cn}}

\keywords{semi-convex, interior estimates, Liouville problem}
\subjclass[2020]{35J60, 53C42}


\begin{abstract}
 We derive a concavity inequality for $k$-Hessian operators under the semi-convexity condition. As an application, we establish interior estimates for semiconvex solutions to the $k$-Hessian equations with vanishing Dirichlet boundary conditions and obtain a Liouville-type result. Additionally, we provide new and simple proofs of Guan-Ren-Wang's work \cite{GRW} on global curvature estimates for $k$-curvature equations.
\end{abstract}

\maketitle

\section{introduction}\label{sec:1}
In this paper, we consider the following $k$-Hessian equations with Dirichlet boundary
\begin{equation}\label{eq1}
\left\{\begin{aligned}
\s_k(D^2 u)&= f(x,u, \nabla u) \quad &\text{in} \ &\Omega, \\
u&=0 \quad&\text{on} \ &\partial \Omega,
\end{aligned}\right.
\end{equation}
where $\Omega$ is a bounded domain in $\R^n$, $f$ is a positive function in $\overline{\Omega}\times R\times R^n$, $u$ is a function defined in $\overline{\Omega}$. Here we denote by $\nabla u$ and $D^2 u$ the gradient of $u$ and the Hessian of $u$, respectively. 
$\s_k$ is the $k$-th elementary symmetric polynomial of $\lambda=(\lambda_1,\dots,\lambda_n)\in \R^n$, which is defined as 
\begin{align*}
    \s_k(x)=\sum_{1\leqslant i_1<i_2...<i_k\leqslant n}\lambda_{i_1}\lambda_{i_2}\cdots\lambda_{i_k},\quad\forall i=1,\dots,n.
\end{align*}
Let $\lambda$ be the eigenvalues of $D^2 u$. Then the $k$-Hessian operator of $u$ is defined by $\s_k(D^2 u)=\s_k(\lambda(D^2))$.

Research on the $k$-Hessian equation 
\begin{equation}\label{eq11}
\s_k(D^2 u)= f(x,u, \nabla u)    
\end{equation}
is a vital area within PDE and geometric analysis. We note that \eqref{eq11} simplifies to a semilinear equation when $k=1$. For $2\leqslant k\leqslant n$, establishing second-order a priori estimates is essential to prove the existence of solutions to these fully nonlinear equations. Pogorelov was the first to investigate the interior $C^2$ estimates for the Monge-Amp\'{e}re equations, as referred to in \cites{GT,Pog}. For general $k$, Chou-Wang \cite{CW} established the Pogorelov estimates for $k$-convex solutions of the $k$-Hessian equation 
\begin{equation}\label{eq12}
\left\{\begin{aligned}
\s_k(D^2 u)&= f(x,u) \quad &\text{in} \ &\Omega\\
u&=0 \quad&\text{on} \ &\partial \Omega,
\end{aligned}\right.
\end{equation}
where $f$ is independent of $\nabla u$. They proved that 
\begin{align*}
 (-u)^{\b}\Delta u\leqslant C
\end{align*}
for some constants $\b>1$ and $C>0$ depending on $n$, $k$, $\inf f$, $\|f\|_{C^{1,1}}$, $\|u\|_{C^1}$ but independent of $\Omega$. Following \cite{CNS}, we define the $k$-convex function by :
For a domain $\Omega\in \R^n$, a function $u\in C^2(\Omega)$ is called k-convex (or $k$-admissible) if the eigenvalues $\lambda(X)$ of Hessian $D^2 u$ are in $\Gamma _k$ (defined by \eqref{gk}) for all $X \in\Omega$. We note that for all $k$-convex functions $u$, the $k$-Hessian operator is elliptic, i.e., 
\begin{align*}
    \dot{\s}_k^{ij}=\frac{\partial\s_k}{\partial u_{ij}}
\end{align*}
is positive-definite. In the following text, we consider the $k$-Hessian equation in this realm.

We extend Theorem 1.5 by Chou-Wang \cite{CW} to cases where $f$ depends on $\nabla u$, under an additional convexity assumption for the function $u$. We refer to the function $u$ as semi-convex if there exists a constant $A>0$ such that the eigenvalues $\lambda$ of $D^2 u$ satisfy
\begin{align*}
 \lambda_i(D^2 u(x))\geqslant -A, \quad \text{for all } x \in \Omega
\end{align*}
for any  $1\leqslant i\leqslant n$.
We have obtained the following interior estimates for semiconvex solutions of Equation \eqref{eq1}.
\begin{thm}\label{thm sc}
Assume that $f\in C^2(\Omega\times R\times R^n)$ satisfies $f>f_0>0$ for a constant $f_0>0$ in $\overline{\Omega}$. Then, for any semi-convex $k$-admissible solution of the Dirichlet problem of the equation \eqref{eq1} provided $2<k\leqslant n-1$, we have 
\begin{align*}
    (-u)^{\a}\Delta u\leqslant C
\end{align*}
where $\a>0$ and $C>0$ depend on $n$, $k$, $f_0$, $\|f\|_{C^{1,1}}$, $\|u\|_{C^1}$ but are independent of $\Omega$.
\end{thm}
\begin{rem}
    The interior estimates in Theorem \ref{thm sc} were derived by Li-Ren-Wang \cite{LRW} under the assumption that $u$ is in the $(k+1)$-convex cone. Tu proved the interior estimates in Theorem \ref{thm sc} using a test function that involved $|X|^2$ in \cite{Tu}. Both of the results mentioned above depend on the diameter of the domain $\Omega$. In the proof of Theorem \ref{thm sc}, we continue with the method in \cite{CW} by Chou-Wang and have obtained estimates independent of $\Omega$. 
\end{rem}
Using this class of interior estimates, we further derive a rigidity theorem for entire solutions in $\R^n$ to
\begin{equation}\label{sk, rig}
    \s_k(D^2 u)=1.
\end{equation}
It has long been a question of concern whether an entire function $u$ in $n$-dimensional Euclidean space satisfying \eqref{sk, rig} is quadratic. For $k=1$, the Liouville property of harmonic functions implies rigidity results for solutions with bounded Hessian. For $k=n$, the Jörgens–Calabi–Pogorelov theorem confirms the rigidity results for all convex entire solutions. Cheng and Yau \cite{CYau} provided an alternative geometric proof.
 For $k=2$, Chang and Yuan \cite{CY} proved that the entire solution to $2$-Hessian equation \eqref{sk, rig} in $\R^n$ is quadratic if
\begin{align*}
    D^2 u\geqslant \d-\sqrt{\frac{2}{n(n-1)}}I
\end{align*}
for any $\d>0$. For $2\leqslant k\leqslant n-1$, the authors also guess that the rigidity theorem holds true under the semi-convexity assumption
\begin{align*}
D^2 u\geqslant -AI
\end{align*}
    with arbitrarily large constant $A>0$.  Warren \cite{War} constructed a class of nontrivial solutions to \eqref{sk, rig}, which are neither semiconvex nor exhibit quadratic growth. Shankar and Yuan \cite{SY} obtained the rigidity theorem for $k=2$ under the semi-convexity assumption. For a general $k$, the rigidity problem for the $k$-Hessian equation remains open.
A function $u$ is said to have quadratic growth if there exists constants $b,c>0$ and sufficient large $R>0$, such that the following inequality holds: 
\begin{align}
u(x)\geqslant c|x|^2-b, \quad \text{for } \ |x|\geqslant R. 
\end{align}
In this paper, we confirm Chang-Yuan's conjecture under the quadratic growth condition for general $k$.
\begin{thm}\label{thm rig}
Given $2<k\leqslant n-1$, let $u$ be any semi-convex entire function in $\R^n$ with $\lambda(D^2 u)\in \Gamma_k$ that satisfies
\begin{align}
    \s_k(D^2u)=1
\end{align} 
and exhibits quadratic growth. Then $u$ is a quadratic polynomial.
\end{thm}
 Assuming a quadratic growth condition, Bao-Chen-Guan-Ji \cite{BCGJ} showed that strictly convex entire solutions to \eqref{sk, rig} are quadratic, Li-Ren-Wang \cite{LRW} extended the convexity assumption to $(k+1)$-convexity,  and Chu-Dinew \cite{CD} further relaxed the condition to
    \begin{align}
        \s_{k+1}>-A
    \end{align}
    for some $A>0$ (see also \cite{Tu}).
In the following text, we note that using the lemma \ref{sk+1 -A} if $\s_{k+1}>-A$, then $\lambda_i>-C(n,k,\s_k,A)$, which means that Theorem \ref{thm rig} implies all these results.\\

To establish the curvature estimates and derive the rigidity theorem, the main difficulty to overcome is how to deal with the third-order derivatives. It is well known that $\s_k^{\frac{1}{k}}$ is concave in the cone $\Gamma_k$, which is very useful to deal with these terms. It should be noted that Ren and Wang \cite{RW1} established the concavity inequality of the Hessian operator for $k=n-1$. They proved that 
\begin{thm} \label{RW n-1}\cite{RW1}
Given two positive constants $\d<1$, $\e$, for any index $i$, if $\k_i\geqslant \d\k_1$ and the positive constant $K$ depending only on $\d$ and $\e$ is sufficient large, then we have
\begin{equation}\label{e1.1}
\k_i[K(\sigma_{n-1})_i^2-\sigma_{n-1}^{pp,qq}h_{ppi}h_{qqi}]-\sigma^{ii}_{n-1}h_{iii}^2+(1+\varepsilon)\sum_{j\neq i}\sigma_{n-1}^{jj}h_{jji}^2\geqslant 0.
\end{equation}
\end{thm}
 In \cite{RW1}, they also pointed out that the inequality \eqref{e1.1} is not true even if $i=1$, $\s_k=1$ and $k_1$ are large enough for $k<n-1$. In addition, Guan and Qiu \cite{GQ} established the interior $C^2$ estimates of the $\sigma_2$-Hessian equation under the assumption $\s_3\geqslant -A$ for some constant $A>0$. In their work, it's crucial to observe that under the condition $\s_3\geqslant -A$ and $\k_1$ large enough
\begin{align}\label{Chen}
-\sum_{i\neq j}u_{ii1}u_{jj1}+C\frac{(\dot{\s}_2^{ii}u_{ii1})^2}{\s_2}\geqslant(1+\d_0)\frac{\dot{\s}_2^{11}u_{111}^{2}}{u_{11}}
\end{align}
for some constant $C>0$ and $\d_0>0$. \eqref{Chen} is basically the similar form of \eqref{e1.1} and is derived by using Chen's \cite{Chen} work on the optimal concavity of the $\s_2$ operators with the eigenvector decomposition method. We note that the method makes sense only for $k=2$ since the Hessian operator is a known matrix. It is difficult to generalize this method to other cases. Moreover, it reminds us that the concavity inequality of the general $k$-Hessian operator may also hold under some additional convexity assumption. we develop a method based on \cite{GRW} by Guan-Ren-Wang and \cite{HZ} by Hong-Zhang which is grounded in the crucial observation of the concavity of $q_k=\frac{\s_k}{\s_{k-1}}$, introduced by Huisken-Sinestrari in \cite{HS}. Under the semi-convexity condition, we derive a concavity inequality which plays a crucial role in establishing curvature estimates of the Dirichlet problem of $k$-Hessian equations \eqref{eq1} and \eqref{eq rig}.
\begin{lem}\label{iq sc}
   Assume $\lbrace\lambda_i\rbrace\in\Gamma_k,$ $\lambda_1\geqslant\lambda_2\cdots\geqslant \lambda_n,$ and $\lambda_n>-A$ for a constant $A>0$. Then there exists some constant $C>0$ depending on $n$, $k$, $\s_k$, $A$ and small constant $\d_0>0$ depending on $k$ such that if $\lambda_1\geqslant C$, then the following inequality holds
    \begin{align}\label{main,iq}
        -\sum_{p\neq q}\frac{\s_k^{pp,qq}\xi_p\xi_q}{\s_k}+K\frac{(\sum_i\s_k^{ii}\xi_i)^2}{\s_k^2}+2\sum_{i>1}\frac{\s_k^{ii}\xi_i^2}{ (\lambda_1+A+1)\s_k}\geqslant (1+\d_0)\frac{\dot{\s}_k^{11} \xi_1^2}{\lambda_1\s_k}
    \end{align}
    for some sufficient large $K>0$ (depending on $\d_0$) where $\xi=( \xi_1,\cdots,\xi_n)$ is an arbitrary vector in $\mathbb{R}^n$.
\end{lem}
The lemma \ref{sk+1 -A} implies that if $\s_{k+1}>-A,$ then $\lambda_i>-C(n,k,\s_k,A)$. In other words, the semiconvexity of $\lbrace\lambda_i\rbrace$ is a weaker condition, which means that Lemma \ref{iq sc} is useful in deriving Guan-Qiu's \cite{GQ} interior curvature estimates. In addition, Lemma \ref{iq sc} removes the additional assumption that $|\k_n|\geqslant \d_0$ for some constant $\d_0>0$ from Hong and Zhang \cite{HZ} employing a simplified and concise approach.

Using the proof in Lemma \ref{iq sc}, we also improve Lemma 3.1 from Lu \cite{Lu2} and present it as Lemma \ref{lem imp}. In \cite{Lu2}, the concavity inequality, Lemma 3.1, plays a crucial role in obtaining curvature estimates for admissible solutions to the problem of a prescribed
curvature measure type in a hyperbolic space (referencing \cite{Yang}). We present Lemma 3.1 in \cite{Lu2} here for ease of comparison.
\begin{lem}[\cite{Lu2}]\label{Lu}
Let $\lambda=(\lambda_1,\cdots,\lambda_n)\in \Gamma_k$ with $\lambda_1\geqslant\cdots\geqslant\lambda_n$ and let $1\leqslant l<k$. For any $\epsilon,\delta,\delta_0\in (0,1)$, there exists a constant $\delta^\prime>0$ depending only on $\epsilon,\delta,\delta_0,n,k$ and $l$ such that if $\lambda_l\geqslant \delta \lambda_1$ and $\lambda_{l+1}\leqslant \delta^\prime\lambda_1$, then we have
\begin{align*}
-\sum_{p\neq q}\frac{\sigma_k^{pp,qq}\xi_p\xi_q}{\sigma_k}+\frac{\left( \sum_i \sigma_k^{ii}\xi_i \right)^2}{\sigma_k^2}\geqslant (1-\epsilon)\frac{\xi_1^2}{\lambda_1^2}-\delta_0\sum_{i>l}\frac{\sigma_k^{ii}\xi_i^2}{\lambda_1 \sigma_k},
\end{align*}
where $\xi=(\xi_1,\cdots,\xi_n)$ is an arbitrary vector in $\mathbb{R}^n$.
\end{lem}
In fact, it is clear that if $\lambda_n>-A$ for a constant $A>0$ and $\lambda_1$ is large enough, then there exists some small $\e$ such that $\dot{\s}_k^{11}\geqslant(1-\e)\frac{\s_k}{\lambda_1}$, which implies that Lemma \ref{iq sc} is a development of Lemma \ref{Lu}. In addition, Lemma \ref{Lu} is not enough to prove Theorem \ref{thm rig}. In \cite{Tu}, Tu employed Lemma \ref{Lu} to obtain the rigidity results under a stronger condition $\s_{k+1}>-C$ for some positive constant $C$. In fact, in most cases $\dot{\s}_k^{11}\gg\frac{\s_k}{\lambda_1}$, which can be seen from the proof of Lemma \ref{sk+1 -A}. This implies that Lemma \ref{iq sc} is stronger than Lemma \ref{lem imp}.

Moreover, Lemma \ref{iq sc} can be regarded as an extension of Lemma 3.3 by Chu \cite{Chu} under a weaker convexity assumption. By applying Lemma 3.3 from \cite{Chu}, Chu gave a simple proof of Theorem 1.1 by Guan, Ren and Wang \cite{GRW}, which established the curvature estimates for convex hypersurfaces satisfying the Weingarten curvature equation
\begin{equation}\label{eq c}
    \s_k(\k(X))=f(X,\nu(X)),\quad \forall X\in M^n,
\end{equation}
where $f\in C^2(\Gamma)$ is a positive function and $\Gamma$ is an open neighborhood of the unit normal bundle of $M$ in $\R^{n+1}\times\mathbb{S}^n$, $\k(X)$ denotes the principal curvatures, $\nu(X)$ is the unit outer normal vector at point $X$, and $\s_k$ represents the $k$-th mean curvature. 

In Remark 4.7 of Guan-Ren-Wang \cite{GRW}, they also deduced the following curvature estimates under a weaker convexity condition. Let $M^n\in \R^{n+1}$ be a closed hypersurface. We say a hypersurface $M$ is semi-convex if there exists a constant $A>0$ such that its principal curvatures satisfy
\begin{align*}
 \k_i(X)\geqslant -A,\quad 1\leqslant i\leqslant n, \quad \forall X\in M.
\end{align*}
\begin{thm}[Guan-Ren-Wang, \cite{GRW}]\label{thm  wsc}
    For any closed semi-convex hypersurface $M^n$ satisfying the curvature equation \eqref{eq c}, there exists a constant $C>0$ depending on $n$, $k$, $\|M\|_{C^1}$, $\inf f$ and $\|f\|_{C^2}$ such that 
    \begin{align*}
        \max_{X\in M}\k_i(X)\leqslant C.
    \end{align*}
\end{thm}

In the specific case where $k=2$, they completely resolved the problem.
\begin{thm}[Guan-Ren-Wang, \cite{GRW}]\label{thm 2d}
For the solution of the curvature equation \eqref{eq c} with $k=2$, it holds that
   \begin{align*}
        \max_{X\in M}\k_i(X)\leqslant C
    \end{align*}
where $C>0$ is a constant that depends on $n$, $k$, $\|M\|_{C^1}$, $\inf f$ and $\|f\|_{C^2}$.
\end{thm}
As an application of Lemma \ref{iq sc}, we also offer new simple proofs of Theorem \ref{thm  wsc} and Theorem \ref{thm 2d}. As mentioned in Remark 4.7 of \cite{GRW}, the curvature estimates in Theorem \ref{thm  wsc} can be obtained by considering the test function
\[\log P(\k(X))+g(u)+a|X|^2\]
where $P$ is a symmetric polynomial of the principal curvatures. It should be noted that Spruck-Xiao \cite{SX} provided a simplified proof of Theorem \ref{thm 2d}. Chu proved Theorem \ref{thm wsc} more simply by assuming further that $M^n$ is a closed convex hypersurface. In both of these works, the maximum principle was applied to a quantity involving the maximal principal curvature rather than the symmetric curvature function $P(\k)$. In what follows, we continue with this approach and use Lemma \ref{iq sc} to prove these results.

The paper is organised as follows. In Section \ref{sec:2}, we collect and prove some algebraic properties of symmetric polynomials. In Section \ref{sec:3}, we prove the crucial lemma concerning the superior concavity of the $k$ -Hessian operator. In Section \ref{sec:4}, we derive the Pogorelov estimates for the $k$-Hessian equation with vanishing Dirichlet boundary and complete the proof of Theorem \ref{thm sc}. In Section \ref{sec:5}, we prove the rigidity theorem and complete the proof of Theorem \ref{thm rig}. In Section \ref{sec:6}, we provide a new and simple proof of Guan-Ren-Wang's famous work \cite{GRW} by using the Lemma \ref{iq sc} derived in Section \ref{sec:3}. 
\begin{ack}
The author would like to thank Prof. Haizhong Li for bringing many related works to our attention regarding this problem. The author also expresses great gratitude to Prof. Yuguang Shi for his ongoing support and encouragement.The author is also deeply grateful to Prof. Weisong Dong for his helpful feedback. The author is supported by the National Key R$\&$D Program of China 2020YFA0712800, the China Postdoctoral Science Foundation No.2023M740108, and the Postdoctoral Fellowship Program of CPSF under Grant Number GZC20230094.
\end{ack}

\section{Preliminaries}\label{sec:2}
\subsection{Symmetric function}
Here we state some algebraic properties of the elementary symmetric function $\s_m(\lambda)$, $m=1,\cdots,n$, where $ \lambda=( \lambda_1,\cdots, \lambda_n)$. We recall $\s_m(\lambda)$ defined by
	\begin{equation*}
	\s_m(\lambda)=\sum_{1\leqslant i_1<\cdots<i_m \leqslant n}\lambda_{i_1}\cdots\lambda_{i_m}, \quad m=1,2, \cdots,n.
	\end{equation*}
	The Garding cone $\Gamma_m$ is an open symmetric convex cone in $\mathbb{R}^n$ with vertex at the origin, given by
	\begin{align}\label{gk}
	\Gamma_m=\lbrace\f \lambda_1,\cdots, \lambda_n\r\in\mathbb{R}^n |\s_j(\lambda)>0,  \forall j= 1,\cdots,m\rbrace.
	\end{align}
	Clearly $\s_m(\lambda)=0$ for $\lambda \in\partial \overline{\Gamma}_m$ and
	\begin{align*}
	\Gamma_n\subset\cdots\subset {\Gamma}_m\subset\cdots\subset\Gamma_1.
	\end{align*}
	In particular, $\Gamma^+ = \Gamma_n$ is called the positive cone,
	\begin{align*}
	\Gamma^+=\lbrace\f \lambda_1,\cdots, \lambda_n\r\in\mathbb{R}^n |\lambda_1>0,\cdots,\lambda_n>0\rbrace.
	\end{align*}
	Always assume $\lambda_1\geqslant\lambda_2\geqslant\cdots\geqslant \lambda_n$. We collect some properties of $\s_k(\lambda)$ where $\lambda\in\Gamma_k$:
 \begin{enumerate}[(1)]
     \item\label{P1} $\sum_{p=1}^n\frac{\partial \s_k}{\partial \lambda_p}\lambda_p^2=\s_1\s_k-(k+1)\s_{k+1}$. 
     \item\label{P2} $\sum_{p=1}^n \frac{\partial \s_k}{\partial \lambda_p}=(n-k+1)\sigma_{k-1}.$
     \item\label{lem'} $ \frac{\partial \s_k}{\partial \lambda_n}\geqslant \cdots\geqslant \frac{\partial \s_k}{\partial \lambda_1}>0.$
     \item\label{fk1} $\frac{\partial \s_k}{\partial \lambda_1}\lambda_1\geqslant \frac{k}{n}\s_k.$
     \item\label{con} $q_k:=\frac{\s_k}{\s_{k-1}}$ is concave, i.e., $\partial^2_{\xi}q_k\leqslant 0$ for any $\xi\in \R^n$.
     \item\label{k1} For any $s<k$, $\s_s>\lambda_1\lambda_2\cdots\lambda_s$.
 \end{enumerate}

 We point out that we will denote $C$, $c$ or $C_i$ and $c_i$ for $i=1,2,\dots,n$ as some positive constants throughout the subsequent sections, and these constants may change from line to line.
\begin{lem}\label{lemn}
   Assume that $\lbrace\lambda_i\rbrace\in \Gamma_k$ and $\lambda_1\geqslant\lambda_2\cdots\geqslant \lambda_n$. Then $\lambda_k>0$ and $|\lambda_n|<(n-k)\lambda_k.$  
 \end{lem}
 \begin{proof}
 From Property \eqref{lem'}, we have
 \begin{align*}
      \frac{\partial \s_k}{\partial\lambda_1\cdots\partial \lambda_{k-1}}=\lambda_k+\cdots+\lambda_n>0.
 \end{align*}
Suppose that only $\lambda_n$ is negative, then $(n-k)\lambda_k>|\lambda_n|$.
 \end{proof}

\begin{lem}\label{sk+1 -A}
    Assume that $\lbrace\lambda_i\rbrace\in \Gamma_k$,  $\lambda_1\geqslant\lambda_2\cdots\geqslant \lambda_n$, and $\s_k(\lambda)$ are bounded from above and below. If $\s_{k+1}>-A$ for any positive constant A, then $\lambda$ is semiconvex, i.e., $\lambda_{\min}>-C(n,k,\s_k, A)$.
\end{lem}
\begin{proof}
There exists constant $C>0$, such that
\begin{align}\label{sk^2}
    \frac{\partial \s_k}{\partial \lambda_n} \lambda_n^2 \leqslant\sum_i \frac{\partial \s_k}{\partial \lambda_i} \lambda_i^2 = \s_1 \s_k -(k+1) \s_{k+1}<C\s_1+(k+1)A.
\end{align}
Then by Property \eqref{lem'} we have 
\begin{align*}
    \sum_i \frac{\partial \s_k}{\partial \lambda_i} \lambda_n^2 <nC\s_1+n(k+1)A.
\end{align*}
Combining $(n-k+1)\s_{k-1}=\sum_i \frac{\partial \s_k}{\partial \lambda_i}$ and Property \eqref{k1}, Lemma \ref{lemn}, there exists constant $C(n,k)>0$
\begin{align*}
C'(n,k)|\lambda_n|^{k+1}\leqslant C(n,k)|\lambda_n|^k\lambda_1<C\lambda_1+n(k+1)A.
\end{align*}
Thus, we get $|\lambda_n|<C(n,k,\s_k, A)$ by Lemma \ref{lemn}, which completes the proof.
\end{proof}
 
 \begin{lem}\label{upperboundforkappak}
     Assume $\lbrace\lambda_i\rbrace\in \Gamma_k$ and $\lambda_1\geqslant\lambda_2\cdots\geqslant \lambda_n$. If $\lambda_n>-A$, then $\lambda_k<C(\s_k, n, A).$
 \end{lem}
 \begin{proof}
     Assume that $\lambda_{k-1}>A$, otherwise $\lambda_k\leqslant A$, then we are done. Since 
     \begin{align*}
        \s_k=\sum_{1\leqslant i_1\leqslant i_2...\leqslant i_k\leqslant n}\lambda_{i_1}\lambda_{i_2}\cdots\lambda_{i_k}>0.
     \end{align*}
     We have 
     \begin{align}\label{kk,n}
         \lambda_1\cdots\lambda_{k-1}\lambda_k\leqslant \s_k-C(n,k)\lambda_1\cdots\lambda_{k-1}\lambda_n.
     \end{align}
     Thus
\begin{align}
          \lambda_k\leqslant & 2\f\frac{\sigma_k}{\lambda_1}\r^{\frac{1}{k-1}}+2C(n,k)|\lambda_n|\label{kk,1}\\
          \leqslant &2\f \frac{n\sigma_k}{\s_1}\r^{\frac{1}{k-1}}+2C(n,k)A\nonumber\\
          \leqslant &c(n,k)\sigma_k^{\frac{1}{k}}+C(n,k,A).\nonumber
     \end{align}
  Here we use property \eqref{k1} in the second inequality and the last inequality follows from Maclaurin's inequality. Thus, there exists some constant $C>0$, such that $\lambda_k\leqslant C$, which depends on the bounds of $\s_k$, $n$ and $A$.
     
 \end{proof}

\begin{lem}\label{sk-1}
    If $\lbrace\lambda_i\rbrace\in \Gamma_k$ and $\lambda_1\geqslant\lambda_2\cdots\geqslant \lambda_n$, then $\s_{k-1}\geqslant C\lambda_1^{\frac{1}{k-1}}$.
\end{lem}
\begin{proof}
Using Newton inequality and Property \eqref{k1}, we have 
    \begin{align}
       \s_{k-1}\geqslant C(n,k)\s_1^{\frac{1}{k-1}}\s_k^{\frac{k-2}{k-1}}\geqslant C\lambda_1^{\frac{1}{k-1}}.
    \end{align}
\end{proof}

 Now we state a crucial calculation by Huisken-Sinestrari \cite{HS}.
\begin{lem}\label{HS}
Given $k\geqslant 2$, if $\lbrace\lambda_i\rbrace\in \Gamma_k$, then for any $\xi\in \R^n$,
\begin{align}\label{q2}
    -\partial^2_{\xi}q_2=\sum_{i=1}^n\frac{\f\xi_i-\frac{\lambda_i}{\s_1}\s_1(\xi)\r^2}{\s_1}    
\end{align}
   and 
   \begin{align}\label{qq2}
       \partial^2_{\xi}q_{k+1}\leqslant \sum_{i=1}^n\frac{\lambda_i^2\partial^2_{\xi}q_{k;i}}{(k+1)(q_{k;i}+\lambda_i)^2}
   \end{align}
   where $q_k=\frac{\s_k}{\s_{k-1}}$,  $q_{k;i}=\frac{\sigma_k(\lambda|i)}{\sigma_{k-1}(\lambda|i)}$, and $\sigma_k(\lambda|i)=\frac{\partial\s_{k+1}}{\partial \lambda_i}$.
\end{lem}
If $\xi\perp \gamma$, then $-\s_1\partial^2_{\xi}q_2\geqslant |\xi|^2$ by \eqref{q2}. More generally, we have the following lemmas that are crucial to the proof of the lemma \ref{iq sc}.

\begin{lem}\label{HS1}
If $\lbrace\lambda_i\rbrace\in \Gamma_k$, then for any $\xi\in \R^n$,
\begin{align}\label{q2corollary1-1}
   -\partial^2_\xi q_{2;i_1,\cdots,i_{k-2}}\geqslant \frac{|[\xi]_{i_1,\cdots,i_{k-2}}^\perp|^2}{\s_1(\lambda|i_1,\cdots,i_{k-2})}.
\end{align}
Here for any $1\leq \ell\leq k-2$, $[\xi]_{i_1,\cdots,i_{k-2}}$ denotes the vector that is obtained by throwing out the $i_1$th,$\cdots,$ $i_{k-2}$th components of $\xi$  and $[\xi]_{i_1,\cdots,i_{k-2}}^\perp$ denotes orthogonal part of $[\xi]_{i_1,\cdots,i_{k-2}}$ with respect to $[\lambda]_{i_1,\cdots,i_{k-2}}.$
Moreover,
\begin{align}\label{qq2corollary1}
       -\partial^2_{\xi}q_{k-1;i_1}\geqslant -\sum_{i=\{1,\cdots, n\}\setminus i_1}\frac{\lambda_i^2\partial^2_{\xi}q_{k-2;i_1,i}}{(k-1)(q_{k-1;i_1,i}+\lambda_i)^2}
   \end{align}   
\end{lem}

\begin{proof}
From \eqref{q2} in Lemma \ref{HS} it follows
\begin{align}\label{q2corollary1}
   -\partial^2_\xi q_{2;i_1,\cdots,i_{k-2}}=\sum_{i=\{1,\cdots,n\}\setminus\{i_1,\cdots,i_{k-2}\}}    \frac{\f\xi_i-\frac{\lambda_i}{\s_1(\lambda|i_1,\cdots,i_{k-2})}\s_1(\xi|i_1,\cdots,i_{k-2})\r^2}{\s_1(\lambda|i_1,\cdots,i_{k-2})}.
\end{align}
Then \eqref{q2corollary1-1} holds by decomposing $[\xi]_{i_1,\cdots,i_{k-2}}$ with respect to $[\lambda]_{i_1,\cdots,i_{k-2}}$, that is, write $[\xi]_{i_1,\cdots,i_{k-2}}=c[\lambda]_{i_1,\cdots,i_{k-2}}+[\xi]^\perp_{i_1,\cdots,i_{k-2}}$.

\eqref{qq2corollary1} follows from \eqref{qq2} .
\end{proof}

\begin{lem}
If $\lbrace\lambda_i\rbrace\in \Gamma_k$, then for any $\gamma\in \R^n$,
    \begin{align}\label{sigma_k-1Hessianqk}
       -\s_{k-1}\partial^2_\gamma q_k\geqslant C\sum_{j=1}^{k-1}\lambda_1\cdots\hat{\lambda}_j\cdots\lambda_{k-1}\left|[\gamma]^\perp_{1,\cdots,\hat{j},\cdots,k-1}\right|^2.
   \end{align}
\end{lem}

\begin{proof}

We note that 
\begin{align}
   q_{k;i}+\lambda_i=\frac{\sigma_k}{\s_{k-1}(\lambda|i)}.
\end{align}
By \eqref{qq2}, \eqref{qq2corollary1} and \eqref{q2corollary1-1},  we have
\begin{equation}
  \begin{aligned}
    -\s_{k-1}\partial^2_{\gamma}q_k\geqslant& -\sum_{i_1=1}^n\frac{\lambda_{i_1}^2\partial_{\gamma}^2q_{k-1;i_1}\s_{k-2}(\lambda|i_1)}{k\frac{\s_{k-1}}{\s_{k-2}(\lambda|i_1)}}\\ \geqslant& -\sum_{i_1=1}^n\sum_{i_2=\{1,\cdots,n\}\setminus\{i_1\}}\frac{\lambda_{i_1}^2\lambda_{i_2}^2\partial_{\gamma}^2q_{k-2;i_1,i_2}\s_{k-3}(\lambda|i_1,i_2)}{k(k-1)\frac{\s_{k-1}}{\s_{k-2}(\lambda|i_1)}\frac{\s_{k-2}(\lambda|i_1)}{\s_{k-3}(\lambda|i_1,i_2)}}\\
    \geqslant &\cdots\\
    \geqslant &-\sum_{i_1=1}^n\cdots\sum_{i_{k-2}=\{1,\cdots,n\}\setminus\{i_1,\cdots,i_{k-3}\}}\frac{2\lambda_{i_1}^2\lambda_{i_2}^2\cdots\lambda_{i_{k-2}}^2\partial_{\gamma}^2 q_{2;i_1,i_2,\cdots,i_{k-2}}\s_1(\lambda|i_1,i_2,\cdots,i_{k-2})}{k!\frac{\s_{k-1}}{\s_1(\lambda|i_1,i_2,\cdots,i_{k-2})}}\\
    \geqslant &\sum_{i_1=1}^n\cdots\sum_{i_{k-2}=\{1,\cdots,n\}\setminus\{i_1,\cdots,i_{k-3}\}}\frac{2\lambda_{i_1}^2\lambda_{i_2}^2\cdots\lambda_{i_{k-2}}^2|[\gamma]_{i_1,\cdots,i_{k-2}}^\perp|^2}{k!\frac{\s_{k-1}}{\s_1(\lambda|i_1,i_2,\cdots,i_{k-2})}}.
\end{aligned}  
\end{equation}
Taking $(i_1,i_2,\cdots,i_{k-2})$ from $(1,\cdots,k-1)$ and denote the remaining index by $j$,
by using the facts that $\s_{k-1}\leqslant C(n,k)\lambda_1\cdots\lambda_{k-1}$ and $\s_1(\lambda|i_1,i_2,\cdots,i_{k-2})\geqslant\lambda_{j}$, we obtain
\begin{align}
     -\s_{k-1}\partial^2_{\gamma}q_k\geqslant C\sum_{1\leqslant i_1\leqslant\dots\leqslant i_{k-2}\leqslant n}\lambda_{i_1}\cdots\lambda_{i_{k-2}}\left|[\gamma]^\perp_{i_1,i_2,\cdots,i_{k-2}}\right|^2.
\end{align}
Summing over $j$ completes the proof.
\end{proof}
\begin{lem}\label{decom}
Assume that $\gamma=(0,\gamma_2,\cdots,\gamma_{k-1},\gamma_k,\dots,\gamma_n)$ satisfying 
\begin{align}\label{perp}
(\gamma_k,\dots,\gamma_n)\perp(\lambda_k,\dots,\lambda_n).
\end{align}
If $\lbrace\lambda_i\rbrace\in \Gamma_k$, then
        \begin{align}
       -\s_{k-1}\partial^2_\gamma q_k\geqslant C\frac{\sum_{j=1}^{k-1}\lambda_1\lambda_2\cdots\lambda_{k-1}\lambda_k^2\gamma_j^2}{\lambda_j^3}+C\lambda_1\lambda_2\cdots\lambda_{k-2}\sum_{p=k}^{n}\gamma_p^2.
   \end{align}
\end{lem}
\begin{proof}
Denote by \[V_i=\lbrace(\gamma_i,\gamma_k,\dots,\gamma_n)\in \R^{n-k+2}\rbrace\] for $2\leqslant i\leqslant k-1$ and \[V=\lbrace (0,\gamma_k,\dots,\gamma_n)|(\gamma_k,\dots,\gamma_n)\perp(\lambda_k,\dots,\lambda_n)\rbrace.\]
For convenience, we name $(0,\lambda_k,\dots,\lambda_n)$ as $\b$. Let $\a_i=(\lambda_i,\lambda_k,\dots,\lambda_n)$, $\b_i=(\sin\t_i,\cos\t_i\frac{\b}{|\b|})$ and $\tan\t_i=-\frac{|\b|}{\lambda_i}$. We note that for any $2\leqslant i\leqslant k-1$, $\b_i\perp\a_i$, $\b_i\perp\b$ and \[V_i=V\oplus\mathcal{L}(\beta_i)\oplus\mathcal{L}(\a_i).\] 
Then we can write any $\t_i=(\lambda_i,\lambda_k,\dots,\lambda_n)\in V_i$ as \begin{align}\label{lami}
    \t_i=(0,\gamma_p,\dots,\gamma_n)-\frac{|\b|}{|\a_i|}\gamma_i\b_i+\frac{\lambda_i}{|\a_i|^2}\gamma_i\a_i
\end{align}
where $\gamma_i\in V$.
Plugging \eqref{lami} into \eqref{sigma_k-1Hessianqk}, we obtain
\begin{align}
    -\s_{k-1}\partial^2_\gamma q_k\geqslant& C\sum_{j=2}^{k-1}\lambda_1\cdots\hat{\lambda}_j\cdots\lambda_{k-1}\f\sum_{k\leqslant p\leqslant n}\gamma_p^2+\frac{|\b|^2}{|\a_j|^2}\gamma_j^2\r\\
    \geqslant& C\sum_{j=2}^{k-1}\lambda_1\cdots\hat{\lambda}_j\cdots\lambda_{k-1}\f\sum_{k\leqslant p\leqslant n}\gamma_p^2+c_0\frac{\lambda_k^2}{\lambda_j^2}\gamma_j^2\r\nonumber\\
    \geqslant& C\frac{\sum_{2\leqslant j\leqslant k-1}\lambda_1\lambda_2\cdots\lambda_{k-1}\lambda_k^2\gamma_j^2}{\lambda_j^3}+C\lambda_1\lambda_2\cdots\lambda_{k-2}\sum_{k\leqslant p\leqslant n}\gamma_p^2
\end{align}
where we use Lemma \ref{lemn} in the second inequality.
\end{proof}

For convenience, we perform some straightforward calculations for $q_k$ here. 
\begin{lem}\label{qk,con}
\begin{align}
   \frac{\frac{\partial^2 q_k}{\partial\lambda_p\partial\lambda_q }}{q_k}=\frac{\frac{\partial^2 \s_k}{\partial\lambda_p\partial\lambda_q}}{\s_k}-\frac{\frac{\partial \s_k}{\partial\lambda_p}}{\s_k}\frac{\frac{\partial \s_{k-1}}{\partial\lambda_q}}{\s_{k-1}}-\frac{\frac{\partial \s_k}{\partial\lambda_q}}{\s_k}\frac{\frac{\partial \s_{k-1}}{\partial\lambda_p}}{\s_{k-1}}-\frac{\frac{\partial^2 \s_{k-1}}{\partial\lambda_p\partial\lambda_q}}{\s_{k-1}}+2\frac{\frac{\partial \s_{k-1}}{\partial\lambda_p}\frac{\partial \s_{k-1}}{\partial\lambda_q}}{\s_{k-1}^2}.
\end{align}
\end{lem}
\begin{proof}
 \begin{equation}
    \begin{aligned}
    \frac{\frac{\partial^2 q_k}{\partial\lambda_p\partial\lambda_q}}{q_k}=&\frac{\partial\f\frac{\frac{\partial \s_k}{\partial\lambda_p}}{\s_{k-1}}-\frac{\frac{\partial \s_{k-1}}{\partial\lambda_p}\s_k}{\s_{k-1}^2}\r}{q_k\partial\lambda_q}\\
    =&\frac{ \frac{\frac{\partial^2 \s_k}{\partial\lambda_p\partial\lambda_q}}{\s_{k-1}}-\frac{\frac{\partial \s_k}{\partial\lambda_p}\frac{\partial \s_{k-1}}{\partial\lambda_q}}{\s_{k-1}^2}-\frac{\frac{\partial \s_k}{\partial\lambda_q}\frac{\partial \s_{k-1}}{\partial\lambda_p}}{\s_{k-1}^2}-\frac{\frac{\partial^2 \s_{k-1}}{\partial\lambda_p\partial\lambda_q}\s_k}{\s_{k-1}^2}+2\frac{\frac{\partial \s_{k-1}}{\partial\lambda_p}\frac{\partial \s_{k-1}}{\partial\lambda_q}\s_k}{\s_{k-1}^3}}{q_k}\\
    =&\frac{\frac{\partial^2 \s_k}{\partial\lambda_p\partial\lambda_q}}{\s_k}-\frac{\frac{\partial \s_k}{\partial\lambda_p}}{\s_k}\frac{\frac{\partial \s_{k-1}}{\partial\lambda_q}}{\s_{k-1}}-\frac{\frac{\partial \s_k}{\partial\lambda_q}}{\s_k}\frac{\frac{\partial \s_{k-1}}{\partial\lambda_p}}{\s_{k-1}}-\frac{\frac{\partial^2 \s_{k-1}}{\partial\lambda_p\partial\lambda_q}}{\s_{k-1}}+2\frac{\frac{\partial \s_{k-1}}{\partial\lambda_p}\frac{\partial \s_{k-1}}{\partial\lambda_q}}{\s_{k-1}^2}.
\end{aligned}
\end{equation}   
\end{proof}

\section{A concavity inequality}\label{sec:3}
In this section, we define
\begin{align*}
\dot{\s}_k^{pp}(\lambda)=\frac{\partial \s_k}{\partial\lambda_p},\quad\ddot{\s}_k^{pp,qq}(\lambda)=\frac{\partial^2 \s_k}{\partial\lambda_p\partial\lambda_q}.
\end{align*}
Subsequently, we provide the proof of the crucial lemma \ref{iq sc}.
\begin{proof}
   First we note that $\log\s_k$ is concave. From Property \eqref{con}, we have for any vector $\xi=( \xi_1,\cdots,\xi_n)$ in $\mathbb{R}^n$
    \begin{equation}\label{logsk}
   \begin{aligned}
-\partial^2_{\xi}\log\s_k=&-\partial^2_{\xi}\left(\log q_k+\log q_{k-1}+\cdots+\log q_2+\log\s_1\right)\\
=&-\frac{\partial^2_{\xi}q_k}{q_k}+(\partial_{\xi}\log q_k)^2-\frac{\partial^2_{\xi}q_{k-1}}{q_{k-1}}+(\partial_{\xi}\log q_{k-1})^2+\cdots+(\partial_{\xi}\log\s_1)^2\\
\geqslant &\sum_{i=1}^{k}(\partial_{\xi}\log q_i)^2\geqslant 0.
    \end{aligned}     
    \end{equation}
    Besides, we denote by $\hat{\xi}=(0,\xi_2,\dots,\xi_n)$ and choose $K\geqslant \frac{1}{\e}$, then  
  \begin{eqnarray}\label{xi1}
     &&-\partial^2_{\xi}\log\s_k+(K-1)\frac{(\sum_i\s_k^{ii}\xi_i)^2}{\s_k^2}\nonumber\\
&\geqslant&-\frac{\partial^2_{\xi}q_k}{q_k}+(\partial_{\xi}\log \s_k-\partial_{\xi}\log\s_{k-1})^2 -\partial^2_{\xi}\log\s_{k-1}+\f\frac{1}{\e}-1\r(\partial_{\xi}\log \s_k)^2\nonumber\\    &\geqslant&-\frac{\partial^2_{\xi}q_k}{q_k}+(1-\epsilon)(\partial_{\xi}\log \s_{k-1})^2-\partial^2_{\xi}\log\s_{k-1}\nonumber\\
&=&-\frac{\partial^2_{\xi}q_k}{q_k}+(1-\epsilon)\f\frac{\dot{\s}_{k-1}^{11} \xi_1}{\s_{k-1}}+\partial_{\hat{\xi}}\log \s_{k-1}\r^2-\partial^2_{\hat{\xi}}\log\s_{k-1}\nonumber\\
&& +\f\frac{\dot{\s}_{k-1}^{11} \xi_1}{\s_{k-1}}\r^2+2\sum_{p>1}\frac{\dot{\s}_{k-1}^{11}\dot{\s}_{k-1}^{pp}-\ddot{\s}_{k-1}^{11,pp}\s_{k-1}}{\s_{k-1}^2}\nonumber\\
&\geqslant&-\frac{\partial^2_{\xi}q_k}{q_k}+(1-\epsilon)\f\frac{\dot{\s}_{k-1}^{11} \xi_1}{\s_{k-1}}+\sum_{i=1}^{k-1}\partial_{\hat{\xi}}\log q_i\r^2+\sum_{i=1}^{k-1}\f\partial_{\hat{\xi}}\log q_i\r^2\nonumber\\
&& +\f\frac{\dot{\s}_{k-1}^{11} \xi_1}{\s_{k-1}}\r^2+2\sum_{\a>1}\frac{\dot{\s}_{k-1}^{11}\dot{\s}_{k-1}^{\a\a}-\ddot{\s}_{k-1}^{11,\a\a}\s_{k-1}}{\s_{k-1}^2} \xi_1\xi_{\a}\nonumber\\
&\geqslant &-\frac{\partial^2_{\xi}q_k}{q_k}
+\f1+(1-\epsilon)\frac{1}{k}\r\f\frac{\dot{\s}_{k-1}^{11} \xi_1}{\s_{k-1}}\r^2+2\sum_{\a>1}\frac{\dot{\s}_{k-1}^{11}\dot{\s}_{k-1}^{\a\a}-\ddot{\s}_{k-1}^{11,\a\a}\s_{k-1}}{\s_{k-1}^2} \xi_1\xi_{\a}
  \end{eqnarray}
  by choosing $K=\frac{1}{\e}$
  Here we estimate the second and the third inequality in \eqref{xi1} as 
  \begin{align}
  \f\frac{\dot{\s}_{k-1}^{11} \xi_1}{\s_{k-1}}+\sum_{i=1}^{k-1}\partial_{\hat{\xi}}\log q_i\r^2+\sum_{i=1}^{k-1}\f\partial_{\hat{\xi}}\log q_i\r^2\geqslant \frac{1}{k}\f\frac{\dot{\s}_{k-1}^{11} \xi_1}{\s_{k-1}}\r^2.    
  \end{align}
By 
  \begin{equation}\label{sl}
      \s_l=\dot{\s}_l^{11}\lambda_1+\s_l(\lambda|1),
  \end{equation}
  we have 
  \begin{equation}
      \frac{\dot{\s}_{k-1}^{11}}{\s_{k-1}}=\frac{\s_{k-1}-\s_{k-1}(\lambda|1)}{\lambda_1\s_{k-1}}
  \end{equation}
  and 
  \begin{equation}
    |\s_{k-1}(\lambda|1)|\leq C(n,k)\lambda_2\lambda_3\cdots\lambda_k\leqslant C\frac{\s_{k-1}}{\lambda_1}
  \end{equation}
  where we use Property \eqref{k1} in the last inequality.
Thus we obtain 
  \begin{align}\label{k-1,1}
      \frac{\dot{\s}_{k-1}^{11}}{\s_{k-1}}\geqslant \frac{1}{\lambda_1}-\frac{C}{\lambda_1^2}. 
  \end{align}
  Since 
\begin{equation}\label{s}
  \begin{aligned}
   \frac{\f\dot{\s}_{k-1}^{11}\dot{\s}_{k-1}^{\a\a}-\ddot{\s}_{k-1}^{11,\a\a}\s_{k-1}\r \xi_1\xi_{\a}}{\s_{k-1}^2} =&\frac{\left(\s_{k-2}(\lambda|1,\a)^2-\s_{k-1}(\lambda|1,\a)\s_{k-3}(\lambda|1,\a)\right) \xi_1\xi_{\a}}{\s_{k-1}^2},
\end{aligned}  
\end{equation}
for any $1<i\leqslant k-1$, we have 
\begin{align}\label{s11}
    \s_{k-2}(\lambda|1,i)\leqslant C(n,k)\frac{\lambda_1\lambda_2\cdots\lambda_k}{\lambda_1\lambda_i}\leqslant C(n,k)\frac{\s_{k-1}\lambda_k}{\lambda_1\lambda_i}
\end{align}
and 
\begin{align}\label{s12}
   |\s_{k-1}(\lambda|1,i)|\s_{k-3}(\lambda|1,i) \leqslant C(n,k)\frac{\lambda_1\cdots\lambda_{k-1}\lambda_k^2}{\lambda_1\lambda_i}\frac{\lambda_1\cdots\lambda_{k-1}}{\lambda_1\lambda_i}\leqslant C(n,k)\f\frac{\s_{k-1}\lambda_k}{\lambda_1\lambda_i}\r^2.
\end{align}
For any $k\leqslant p\leqslant n$, we have 
\begin{align}\label{s21}
    \s_{k-2}(\lambda|1,p)\leqslant C(n,k)\frac{\lambda_1\lambda_2\cdots\lambda_{k-1}}{\lambda_1}\leqslant C(n,k)\frac{\s_{k-1}}{\lambda_1}
\end{align}
 and 
 \begin{align}\label{s22}
     |\s_{k-1}(\lambda|1,p)|\s_{k-3}(\lambda|1,p)\leqslant C(n,k)\frac{\lambda_1\cdots\lambda_{k-1}\lambda_k}{\lambda_1}\frac{\lambda_1\cdots\lambda_{k-2}}{\lambda_1}\leqslant C(n,k)\frac{\s_{k-1}^2}{\lambda_1^2}.
 \end{align}
 Here we use Property \eqref{k1} and Lemma \ref{lemn} in \eqref{s11}, \eqref{s12}, \eqref{s21} and \eqref{s22}. 
 Then for any $1<i\leqslant k-1$, \eqref{s} becomes 
 \begin{align}\label{s,1}
  -2\frac{\f\dot{\s}_{k-1}^{11}\dot{\s}_{k-1}^{ii}-\ddot{\s}_{k-1}^{11,ii}\s_{k-1}\r \xi_1\xi_{i}}{\s_{k-1}^2}\geqslant -2\frac{C\lambda_k^2| \xi_1\xi_i|}{\lambda_1^2\lambda_i^2}\geqslant -\frac{\e \xi_1^2}{4(n-1)\lambda_1^2}-\frac{4(n-1)C^2\lambda_k^4\xi_i^2}{\e\lambda_1^2\lambda_i^4}.
 \end{align} 
Then for any $k\leqslant p\leqslant n$, \eqref{s} becomes 
 \begin{align}\label{s,2}
 -2\frac{\f\dot{\s}_{k-1}^{11}\dot{\s}_{k-1}^{pp}-\ddot{\s}_{k-1}^{11,pp}\s_{k-1}\r \xi_1\xi_{p}}{\s_{k-1}^2}\geqslant -2\frac{C| \xi_1\xi_p|}{\lambda_1^2}\geqslant-\frac{\e \xi_1^2}{4(n-1)\lambda_1^2}-\frac{4(n-1)C^2\xi_p^2}{\e\lambda_1^2}.
 \end{align}
Here we use Cauchy-Schwartz inequality in \eqref{s,1} and \eqref{s,2}.
 Recall that 
 \begin{align}\label{la 11aa}
\dot{\s}_k^{11}=\lambda_\a\ddot{\s_k}^{11,\a\a}+\s_{k-1}(\lambda|1,\a),
 \end{align}
 \begin{align}  \dot{\s}_k^{\a\a}=\lambda_1\ddot{\s_k}^{11,\a\a}+\s_{k-1}(\lambda|1,\a).   
 \end{align}        
If $\s_{k-1}(\lambda|1,\a)\leqslant0$, then due to \eqref{la 11aa} $\lambda_{\a}>0$ and 
\begin{align}\label{sk-1<0}
\frac{\lambda_{\a}}{\lambda_1}\dot{\s}_k^{\a\a}=\lambda_{\a}\ddot{\s_k}^{11,\a\a}+\frac{\lambda_{\a}}{\lambda_1}\s_{k-1}(\lambda|1,\a)\geqslant\dot{\s}_k^{11}.
\end{align}
If $\s_{k-1}(\lambda|1,\a)>0$, using the Maclaurin's inequality, we have 
\begin{align}\label{maciq}
    \s_{k-1}(\lambda|1,\a)^{\frac{k-2}{k-1}}\leqslant\s_{k-2}(\lambda|1,\a).
\end{align}
Since
\begin{align*}
   \s_{k-1}(\lambda|1,\a)\leqslant C\lambda_2\lambda_3\cdots\lambda_{k-1}\lambda_k\leqslant C\lambda_1^{k-2}C(A),
\end{align*}
multiplying it with \eqref{maciq} we obtain
\begin{align*}
 \s_{k-1}(\lambda|1,\a)\leqslant C\lambda_1^{-\frac{1}{k-1}}(\lambda_1\ddot{\s_k}^{11,\a\a}).   
\end{align*}
Then using Property \eqref{fk1}, we have 
\begin{align}\label{s_k-1>0,>}
(\frac{\lambda_{\a}}{\lambda_1}+C\lambda_1^{-\frac{1}{k-1}})\dot{\s}_k^{\a\a}\geqslant(\frac{\lambda_{\a}}{\lambda_1}+C\lambda_1^{-\frac{1}{k-1}})\lambda_1\ddot{\s_k}^{11,\a\a}\geqslant \dot{\s}_k^{11}\geqslant\frac{k\s_k}{n\lambda_1},\quad \text{for }\lambda_{\a}>0,
\end{align}
and 
\begin{align}\label{s_k-1>0,<}
C\lambda_1^{-\frac{1}{k-1}}\dot{\s}_k^{\a\a}\geqslant(\frac{\lambda_{\a}}{\lambda_1}+C\lambda_1^{-\frac{1}{k-1}})\lambda_1\ddot{\s_k}^{11,\a\a}\geqslant \dot{\s}_k^{11}\geqslant\frac{k\s_k}{n\lambda_1},\quad \text{for }\lambda_{\a}\leqslant 0.
\end{align}
We now divide it into two cases. We assume that there exists some positive constant $C$.

If $\lambda_{\a}\leqslant C\lambda_1^{\frac{k-2}{k-1}}$, then
by \eqref{sk-1<0}, \eqref{s_k-1>0,>} and \eqref{s_k-1>0,<}, we obtain that 
\begin{align}
    C\lambda_1^{-\frac{1}{k-1}}\dot{\s}_k^{\a\a}\geqslant\frac{1}{\lambda_1}.
\end{align}
For any $k\leqslant p\leqslant n$, $|\lambda_p|\leqslant C(A)$ by Lemma \ref{upperboundforkappak}. Then \eqref{s,2} becomes 
 \begin{align}\label{s,2'}
 -2\frac{\f\dot{\s}_{k-1}^{11}\dot{\s}_{k-1}^{pp}-\ddot{\s}_{k-1}^{11,pp}\s_{k-1}\r \xi_1\xi_{p}}{\s_{k-1}^2}\geqslant-\frac{\e \xi_1^2}{4(n-1)\lambda_1^2}-\frac{C\s_k^{pp}\xi_p^2}{(\e\lambda_1^{\frac{1}{k-1}})\lambda_1}.
 \end{align}
For any $1<i\leqslant k-1$ and $\lambda_i\leqslant C\lambda_1^{\frac{k-2}{k-1}}$, \eqref{s,1} becomes 
\begin{align}
  -2\frac{\f\dot{\s}_{k-1}^{11}\dot{\s}_{k-1}^{ii}-\ddot{\s}_{k-1}^{11,ii}\s_{k-1}\r \xi_1\xi_{i}}{\s_{k-1}^2}\geqslant -\frac{\e \xi_1^2}{4(n-1)\lambda_1^2}-\frac{4(n-1)C^2\xi_i^2}{\e\lambda_1^2}\geqslant-\frac{\e \xi_1^2}{4(n-1)\lambda_1^2}-\frac{C\dot{\s}_k^{ii}\xi_i^2}{(\e\lambda_1^{\frac{1}{k-1}})\lambda_1}.
 \end{align} 
 If $\lambda_{\a}\geqslant C\lambda_1^{\frac{k-2}{k-1}}$, then
by \eqref{sk-1<0} and \eqref{s_k-1>0,>}, we obtain that 
\begin{align}
    C\frac{\lambda_i}{\lambda_1}\dot{\s}_k^{ii}\geqslant\frac{\s_k}{\lambda_1}.
\end{align}
For any $1<i\leqslant k-1$ and $\lambda_i\geqslant C\lambda_1^{\frac{k-2}{k-1}}$, \eqref{s,1} becomes 
\begin{align}
  -2\frac{\f\dot{\s}_{k-1}^{11}\dot{\s}_{k-1}^{ii}-\ddot{\s}_{k-1}^{11,ii}\s_{k-1}\r \xi_1\xi_{i}}{\s_{k-1}^2}\geqslant-\frac{\e \xi_1^2}{4(n-1)\lambda_1^2}-\frac{C(A)\dot{\s}_k^{ii}\xi_i^2}{\e\lambda_1^2\lambda_i^3}\geqslant-\frac{\e \xi_1^2}{4(n-1)\lambda_1^2}-\frac{C\dot{\s}_k^{ii}\xi_i^2}{(\e\lambda_1^{\frac{1}{k-1}})\lambda_1}.
 \end{align} 
By direct calculation, \eqref{xi1} implies that
\begin{equation}\label{xi2}
    \begin{aligned}
         &-\sum_{p\neq q}\frac{\s_k^{pp,qq}\xi_p\xi_q}{\s_k}+K\frac{(\sum_i\s_k^{ii}\xi_i)^2}{\s_k^2}+2\sum_{i>1}\frac{\s_k^{ii}\xi_i^2}{ (\lambda_1+A+1)\s_k}\\
         \geqslant &-\frac{\partial^2_{\xi}q_k}{q_k}
+\f1+(1-\epsilon)\frac{1}{k}\r\f\frac{\dot{\s}_{k-1}^{11} \xi_1}{\s_{k-1}}\r^2-\frac{\e}{4}\frac{ \xi_1^2}{\lambda_1^2}+\f2-\frac{C}{\epsilon\lambda_1^{\frac{1}{k-1}}}-\frac{C(A)}{\lambda_1^2}\r\sum_{i>1}\frac{\s_k^{ii}\xi_i^2}{\lambda_1\s_k}\\
\geqslant& -\frac{\partial^2_{\xi}q_k}{q_k}
+\f1+(1-\epsilon)\frac{1}{k}-\e\r\f1-\frac{C_1(A)}{\lambda_1}\r^2\frac{ \xi_1^2}{\lambda_1^2}+\f2-\frac{C_2}{\epsilon\lambda_1^{\frac{1}{k-1}}}-\frac{C_3(A)}{\lambda_1^2}\r\sum_{i>1}\frac{\s_k^{ii}\xi_i^2}{\lambda_1\s_k}
    \end{aligned}
\end{equation}
where we use \eqref{k-1,1} in the last inequality.
We choose $\epsilon=\frac{1}{(k+1)^2}$ and assume that $\lambda_1>\max\left\lbrace (k+3)C_1(A),\frac{C_3(A)}{C_2(k+1)^2}\right\rbrace.$
Then \eqref{xi2} becomes 
\begin{equation}\label{iq c0}
    \begin{aligned}
         &-\sum_{p\neq q}\frac{\s_k^{pp,qq}\xi_p\xi_q}{\s_k}+K\frac{(\sum_i\s_k^{ii}\xi_i)^2}{\s_k^2}+2\sum_{i>1}\frac{\s_k^{ii}\xi_i^2}{ (\lambda_1+A+1)\s_k}\\
\geqslant& \frac{k+2}{k+1}\f1-\frac{C_1(A)}{\lambda_1}\r^2\frac{ \xi_1^2}{\lambda_1^2}+\f2-\frac{C_2}{\epsilon\lambda_1^{\frac{1}{k-1}}}-\frac{C_3(A)}{\lambda_1^2}\r\sum_{i>1}\frac{\s_k^{ii}\xi_i^2}{\lambda_1\s_k}\\
\geqslant& \frac{(k+2)^2}{(k+1)(k+3)}\frac{ \xi_1^2}{\lambda_1^2}+\f2-2(k+1)^2\frac{C_2}{\lambda_1^{\frac{1}{k-1}}}\r\sum_{i>1}\frac{\s_k^{ii}\xi_i^2}{\lambda_1\s_k}.
    \end{aligned}
\end{equation}
Recall that 
\begin{align}
    \s_k=\lambda_1\dot{\s}_k^{11}+\s_k(\lambda|1).
\end{align}
We note that if 
\begin{align}
    \s_k(\lambda|1)\geqslant-\frac{\s_k}{2(k+2)^2-1},
\end{align}
then \eqref{iq c0} implies the conclusion \eqref{main,iq} by assuming $\lambda_1\geq \f C_2(k+1)^2\r^{k-1}$ and choosing $\d_0=\frac{1}{2(k+1)(k+3)}$.

In the subsequent content, we denote by $c_0=\frac{1}{2(k+2)^2-1}$ and suppose that
\begin{align}\label{sk1 c0}
    \s_k(\lambda|1)<-c_0\s_k.
\end{align}
Under this assumption, we have the following claims.

\emph{Claim 1:}\[\lambda_k\leqslant C(n,k)|\lambda_n|.\]
From \eqref{kk,n}, 
\begin{equation}\label{k1,n}
    \lambda_1\cdots\lambda_k\leqslant 2\s_k 
\end{equation}
or 
\begin{equation}
    \lambda_1\cdots\lambda_k\leqslant 2 C(n,k)\lambda_1\cdots\lambda_{k-1}|\lambda_n|.
\end{equation}
If we assume that \eqref{k1,n} holds, then 
\begin{align}\label{sk|1}
    |\s_k(\lambda|1)|\leqslant C(n,k)\lambda_2\cdots\lambda_k^2\leqslant C'(n,k,A)\frac{\s_k}{\lambda_1}<c_0\s_k
\end{align}
by assuming $\lambda_1>\frac{C'(n,k,A)}{c_0}$, which is contradict to our assumption above. Thus, we complete the proof of the first claim.

\emph{Claim 2:} there exists some $c_1,c_2>0$ such that $-\s_{k+1}\geqslant c_1\lambda_1\cdots\lambda_{k-1}\lambda_k^2\geqslant c_2\lambda_1$.

By Claim 1, \eqref{sk|1} and \eqref{sk1 c0}, we have
\begin{align}\label{l1 c2}
\lambda_1\cdots\lambda_{k-1} \lambda_n^2\geqslant c\lambda_1\cdots\lambda_{k-1} \lambda_k^2\geqslant \frac{cc_0\s_k}{C(n,k)}\lambda_1
\end{align}
and denote $c_2=\frac{cc_0}{C(n,k)}$.
Since $\s_k(\lambda|1)<0$, we have $\lambda_n<0$ and from \eqref{sl}
\begin{align}
   \s_k(\lambda|n) =\s_k-\lambda_n\dot{\s}_k^{nn}>0.
\end{align}
Using Property \eqref{k1}, we get 
\begin{align}\label{sk,n}
    \dot{\s}_k^{nn}>\lambda_1\cdots\lambda_{k-1}.
\end{align}
We suppose that there exist some large $M>0$ (to be chosen later) and $1<l\leqslant k-1$, such that $\lambda_1> M^k$, $\lambda_l> M$ and $\lambda_{l+1}\leqslant M$ in the following proof.
By \eqref{sl}, \eqref{sk,n} and Property \eqref{P1}, we have
\begin{eqnarray}\label{sk+1}
   \begin{aligned}
\lambda_1\cdots\lambda_{k-1} \lambda_n^2\leqslant&\dot{\s}_k^{nn}\lambda_n^2\leqslant\sum_{p=l+1}^n\dot{\s}_k^{pp}\lambda_p^2\\
=& \sum_{i=1}^{l}\lambda_i(\lambda_i\dot{\s}_k^{ii}+\s_k(\lambda|i))+\sum_{p=l+1}^n\lambda_p\s_k-\sum_{i=1}^{l}(\lambda_i\dot{\s}_{k+1}^{ii}+\s_{k+1}(\lambda|i))\\
&-(k+1-l)\s_{k+1}-\sum_{i=1}^{l}\dot{\s}_k^{ii}\lambda_i^2\\
=&\sum_{p=l+1}^n\lambda_p\s_k-\sum_{i=1}^{l}\s_{k+1}(\lambda|i)-(k+1-l)\s_{k+1}.
\end{aligned} 
\end{eqnarray}
We now show that by using Claim 1 for any $1\leqslant i \leqslant l$
\begin{align}\label{sk+1,i}
    |\s_{k+1}(\lambda|i)|\leqslant \frac{C(n,k)\lambda_1\cdots\lambda_{k-1}\lambda_k^3}{\lambda_i}\leqslant \frac{C_1(A)}{M}\lambda_1\cdots\lambda_{k-1}\lambda_n^2.
\end{align}
Applying \eqref{l1 c2}, we have
\begin{align}\label{kpsk}
 \sum_{p=l+1}^n\lambda_p\s_k\leqslant c_1M\leqslant \frac{c_1\lambda_1}{M}\leqslant \frac{c\lambda_1\cdots\lambda_{k-1} \lambda_n^2}{M}.   
\end{align}

Then by choosing $M\geqslant \max\lbrace\frac{\s_k}{c_1},2(1+(k-1)C_1\rbrace)$, we obtain
\begin{align}
  -\s_{k+1}\geqslant \frac{\lambda_1\cdots\lambda_{k-1} \lambda_n^2}{4}  ,
\end{align}
which complete the proof of \emph{Claim 2}.

\emph{Claim 3:} There exists some $c_3,c_4>0$ such that $c_3\frac{\lambda_1\cdots\lambda_k^2}{\lambda_1^2}\leqslant \dot{\s}_k^{11}\leqslant c_4\frac{\lambda_1\cdots\lambda_k^2}{\lambda_1^2}$.

Combining Claim 2, we have 
\begin{align}\label{sk,1+}
-\s_k(\lambda|1)=\frac{-\s_{k+1}-\s_{k+1}(\lambda|1)}{\lambda_1}\leqslant \frac{-\s_{k+1}+C\lambda_2\cdots\lambda_k^3}{\lambda_1}\leqslant \frac{-\s_{k+1}(1+\frac{C}{\lambda_1})}{\lambda_1}.
\end{align}
For the same reason
\begin{align}\label{sk(k1)}
-\s_k(\lambda|1)\geqslant \frac{-\s_{k+1}(1-\frac{C}{\lambda_1})}{\lambda_1}.
\end{align}
Then using the assumption $\s_k(\lambda|1)<-c_0$ and Claim 2, we have
\begin{equation}\label{sk1,l}
    \begin{aligned}
        \dot{\s}_k^{11}=&\frac{\s_k-\s_k(\lambda|1)}{\lambda_1}\geqslant \frac{-\s_k(\lambda|1)}{\lambda_1}\\
        \geqslant&c_4\frac{\lambda_2\cdots\lambda_k^2}{\lambda_1}.
    \end{aligned}
\end{equation}
For the same reason, we have 
\begin{align}\label{sk1,u}
    \dot{\s}_k^{11}
        \leqslant c_3\frac{\lambda_2\cdots\lambda_k^2}{\lambda_1}
\end{align}
which complete the proof of \emph{Claim 3}.

 We denote by $\a=(\lambda_1,\dots,\lambda_n)$, $e_1=(1,0,\dots,0)$ and $\gamma=(0,\gamma_2,\dots\gamma_k,\dots,\gamma_n).$
Now we assume that 
\begin{align}
 \xi=(1+a)\frac{\a \xi_1}{\lambda_1}-a\xi_1e_1+\gamma   
\end{align}
and 
\begin{align}\label{perp 3}
   (\gamma_k,\dots,\gamma_n) \perp(\lambda_k,\dots,\lambda_n).
\end{align}
Then by Lemma \ref{qk,con} the first term of the right side of \eqref{xi2} becomes 
\begin{equation}\label{qk}
\begin{aligned}
    -\frac{\partial^2_{\xi}q_k}{q_k}&=-(-a\xi_1e_1+\gamma)_p\frac{\ddot{q}_k^{pp,qq}}{q_k}(-a\xi_1e_1+\gamma)_q\\
    =&-2a^2\xi_1^2\frac{\dot{\s}_{k-1}^{11}}{\s_{k-1}}\f\frac{\dot{\s}_{k-1}^{11}}{\s_{k-1}}-\frac{\dot{\s}_k^{11}}{\s_k}\r-\frac{\partial^2_{\gamma}q_k}{q_k}\\
&+2a\xi_1\gamma_\a\f\frac{\ddot{\s}_k^{11,\a\a}}{\s_k}-\frac{\dot{\s}_k^{\a\a}}{\s_k}\frac{\dot{\s}_{k-1}^{11}}{\s_{k-1}}-\frac{\dot{\s}_k^{11}}{\s_k}\frac{\dot{\s}_{k-1}^{\a\a}}{\s_{k-1}}+2\frac{\dot{\s}_{k-1}^{11}\dot{\s}_{k-1}^{\a\a}}{\s_{k-1}^2}-\frac{\ddot{\s}_{k-1}^{11,\a\a}}{\s_{k-1}}\r.
\end{aligned}    
\end{equation}
By direct calculation, we obtain following equations:
\begin{equation}
\begin{aligned}
 \frac{\ddot{\s}_k^{11,\a\a}}{\s_k}-\frac{\dot{\s}_k^{\a\a}}{\s_k}\frac{\s_{k-1}^{11}}{\s_{k-1}}=&\frac{\dot{\s}_k^{\a\a}-\s_{k-1}(\lambda|1\a)}{\lambda_1\s_k}-\frac{\dot{\s}_k^{\a\a}}{\s_k}\frac{\s_{k-1}-\s_{k-1}(\lambda|1)}{\lambda_1\s_{k-1}}\\
=&-\frac{\s_{k-1}(\lambda|1\a)}{\lambda_1\s_k}+\frac{\dot{\s}_k^{\a\a}}{\s_k}\frac{\dot{\s}_k^{11}}{\lambda_1\s_{k-1}},
\end{aligned}   
\end{equation}
and
\begin{align}\label{sk-1k-3}
    \frac{\ddot{\s}_{k-1}^{11,\a\a}}{\s_{k-1}}-\frac{\dot{\s}_{k-1}^{\a\a}\dot{\s}_{k-1}^{11}}{\s_{k-1}^2}=\frac{\s_{k-2}^2(\lambda|1\a)-\s_{k-1}(\lambda|1\a)\s_{k-3}(\lambda|1\a)}{\s_{k-1}^2}.
\end{align}
Hence, for any $2\leqslant i\leqslant k-1$, we have
\begin{align}\label{sk-1k-3,i}
    \left|\frac{\ddot{\s}_{k-1}^{11,ii}}{\s_{k-1}}-\frac{\dot{\s}_{k-1}^{ii}\dot{\s}_{k-1}^{11}}{\s_{k-1}^2}\right|\leqslant \frac{C(n,k)\lambda_k^2}{\lambda_1^2\lambda_i^2}\leqslant C\frac{\dot{\s}_k^{11}\lambda_k^2}{\lambda_1\lambda_i^2}.
\end{align}
For any $k\leqslant p\leqslant n$, we have 
\begin{align}\label{sk-1k-3,p}
    \left|\frac{\ddot{\s}_{k-1}^{11,pp}}{\s_{k-1}}-\frac{\dot{\s}_{k-1}^{pp}\dot{\s}_{k-1}^{11}}{\s_{k-1}^2}\right|\leqslant \frac{C(n,k)}{\lambda_1^2}\leqslant C\frac{\dot{\s}_k^{11}}{\lambda_1}.
\end{align}
Here we use Property \eqref{fk1}, Claim 2 and Claim 3 in \eqref{sk-1k-3,i} and \eqref{sk-1k-3,p}.
Furthermore, for any $2\leqslant i\leqslant k-1$, we have
\begin{equation}\label{i}
    \begin{aligned}
     &\frac{\ddot{\s}_k^{11,ii}}{\s_k}-\frac{\dot{\s}_k^{ii}}{\s_k}\frac{\s_{k-1}^{11}}{\s_{k-1}} -\frac{\dot{\s}_k^{11}}{\s_k}\frac{\dot{\s}_{k-1}^{ii}}{\s_{k-1}}+\frac{\dot{\s}_{k-1}^{11}\dot{\s}_{k-1}^{ii}}{\s_{k-1}^2}\\
     =&\frac{-\s_k+\dot{\s}_k^{11}\lambda_1+\s_k(\lambda|1i)}{\lambda_1\lambda_i\s_k}+\frac{\dot{\s}_k^{ii}}{\s_k}\frac{\dot{\s}_k^{11}}{\lambda_1\s_{k-1}}-\f\frac{1}{\lambda_i}
     -\frac{\dot{\s}_k^{ii}}{\lambda_i\s_{k-1}}\r\frac{\dot{\s}_k^{11}}{\s_k}\\
     &+\frac{\f\s_{k-1}-\dot{\s}_k^{11}\r\f\s_{k-1}-\dot{\s}_k^{ii}\r}{\lambda_1\lambda_i\s_{k-1}^2}\\
     =&\frac{\s_k(\lambda|1i)}{\lambda_1\lambda_i\s_k}-\frac{\dot{\s}_k^{ii}\s_k(k|1)}{\lambda_1\lambda_i\s_{k-1}\s_k}+\frac{\dot{\s}_k^{11}}{\lambda_1}\f\frac{\dot{\s}_k^{ii}}{\s_{k-1}\s_k}-\frac{\s_{k-2}(\lambda|i)}{\s_{k-1}^2}\r.
    \end{aligned}
\end{equation}
Here 
\begin{align}\label{i1}
  \left|\frac{\s_k(\lambda|1i)}{\lambda_1\lambda_i\s_k}\right|\leqslant \frac{C\lambda_1\dots\lambda_k^3}{\lambda_1^2\lambda_i^2\s_k}\leqslant\frac{C\dot{\s}_k^{11}\lambda_k}{\lambda_i^2}. 
\end{align}

By using \eqref{sk(k1)}, we have 
\begin{align}\label{i2}
 \left|\frac{\dot{\s}_k^{ii}\s_k(k|1)}{\lambda_1\lambda_i\s_{k-1}\s_k}\right|\leqslant        C\frac{\lambda_1\cdots\lambda_k^2\dot{\s}_k^{11}}{\lambda_i^2\s_{k-1}\s_k}\leqslant C\frac{\lambda_k^2\dot{\s}_k^{11}}{\lambda_i^2}.
\end{align}
Similarly, combining Claim 2 we have
\begin{align}\label{i3}
   \left|\frac{\s_{k-1}(\lambda|i)}{\s_{k-1}\s_k}-\frac{\s_{k-2}(\lambda|i)}{\s_{k-1}^2}\right|\leqslant C\frac{\lambda_1\cdots\lambda_k}{\lambda_i\s_{k-1}\s_k}+C\frac{\lambda_1\cdots\lambda_{k-1}}{\lambda_i\s_{k-1}^2}\leqslant \frac{C\lambda_k}{\lambda_i}\f1+\frac{\lambda_k}{\lambda_1}\r,
   \end{align}
since 
\begin{align}
    \s_{k-1}>\lambda_1\dots\lambda_{k-1}=\frac{\lambda_1\dots\lambda_k^2}{\lambda_k^2}\geq C\frac{\lambda_1}{\lambda_k^2}.
\end{align}
Inserting \eqref{i1}, \eqref{i2} and \eqref{i3} into \eqref{i} and combining \eqref{sk-1k-3,i} we have 
\begin{equation}\label{i,l}
    \begin{aligned}
&2a\xi_1\gamma_i\f\frac{\ddot{\s}_k^{11,ii}}{\s_k}-\frac{\dot{\s}_k^{ii}}{\s_k}\frac{\dot{\s}_{k-1}^{11}}{\s_{k-1}}-\frac{\dot{\s}_k^{11}}{\s_k}\frac{\dot{\s}_{k-1}^{ii}}{\s_{k-1}}+2\frac{\dot{\s}_{k-1}^{11}\dot{\s}_{k-1}^{ii}}{\s_{k-1}^2}-\frac{\ddot{\s}_{k-1}^{11,ii}}{\s_{k-1}}\r\\
\geqslant&-2C|a\xi_1\gamma_i|\frac{\dot{\s}_k^{11}\lambda_k}{\lambda_i^2}\geqslant -\e\frac{\dot{\s}_k^{11}\lambda_1^2}{\lambda_i^3}\gamma_i^2-\frac{Ca^2}{\e}\frac{\dot{\s}_k^{11}\xi_1^2}{\lambda_1^2}
    \end{aligned}
\end{equation}
where we use $\lambda_i>\lambda_k$, $\lambda_1\gg1$ and $\lambda_k<C(A)$ by Lemma \ref{upperboundforkappak} in the last inequality.

For any $k\leqslant p\leqslant n$, we have 
\begin{equation}\label{p,l}
    \begin{aligned}
       &\frac{\ddot{\s}_k^{11,pp}}{\s_k}-\frac{\dot{\s}_k^{pp}}{\s_k}\frac{\s_{k-1}^{11}}{\s_{k-1}}-\frac{\dot{\s}_k^{11}}{\s_k}\frac{\dot{\s}_{k-1}^{pp}}{\s_{k-1}}+\frac{\dot{\s}_{k-1}^{11}\dot{\s}_{k-1}^{pp}}{\s_{k-1}^2}\\ 
       =&\frac{-\s_k+\dot{\s}_k^{pp}\lambda_p+\s_k(\lambda|1p)}{\lambda_1^2\s_k}+\frac{\dot{\s}_k^{pp}}{\s_k}\frac{\dot{\s}_k^{11}}{\lambda_1\s_{k-1}}+\frac{\dot{\s}_{k-1}^{pp}}{\s_{k-1}}\f\frac{\dot{\s}_{k-1}^{11}}{\s_{k-1}}-\frac{\dot{\s}_k^{11}}{\s_k}\r\\
        =&\frac{\s_k(\lambda|1p)-\s_k}{\lambda_1^2\s_k}+\frac{\s_{k-1}\lambda_p-\s_{k-2}(\lambda|p)\lambda_p^2}{\lambda_1^2\s_k}+\frac{\dot{\s}_k^{pp}}{\s_k}\frac{\dot{\s}_k^{11}}{\lambda_1\s_{k-1}}+\frac{\dot{\s}_{k-1}^{pp}}{\s_{k-1}}\f\frac{\dot{\s}_{k-1}^{11}}{\s_{k-1}}-\frac{\dot{\s}_k^{11}}{\s_k}\r.
    \end{aligned}
\end{equation}
By using Claim 2 and Claim 3, we have 
\begin{align}\label{p2}
    \frac{|\s_k(\lambda|1p)-\s_k|}{\lambda_1^2\s_k}\leqslant\frac{C\lambda_2\cdots\lambda_k^2}{\lambda_1^2}+\frac{1}{\lambda_1^2}\leqslant C'\frac{\dot{\s}_k^{11}}{\lambda_1},
\end{align}
\begin{align}\label{p1}
   \frac{\s_{k-2}(\lambda|p)\lambda_p^2}{\lambda_1^2\s_k}\leqslant \frac{ C\lambda_1\cdots\lambda_{k-2}\lambda_k^2}{\lambda_1^2\s_k}\leqslant C'\frac{\dot{\s}_k^{11}}{\lambda_{k-1}},
\end{align}
and 
\begin{align}\label{p3}
   \frac{\dot{\s}_k^{pp}}{\s_k}\frac{\dot{\s}_k^{11}}{\lambda_1\s_{k-1}}\leqslant C\frac{\lambda_1\cdots\lambda_{k-1}}{\s_{k-1}}\frac{\dot{\s}_k^{11}}{\lambda_1}\leqslant C \frac{\dot{\s}_k^{11}}{\lambda_1}.
\end{align}
Besides,
\begin{align}\label{p4}
   \left|\frac{\dot{\s}_{k-1}^{pp}}{\s_{k-1}}\f\frac{\dot{\s}_{k-1}^{11}}{\s_{k-1}}-\frac{\dot{\s}_k^{11}}{\s_k}\r\right|\leqslant C\frac{\lambda_1\cdots\lambda_{k-2}\dot{\s}_k^{11}}{\s_{k-1}}\leqslant C\frac{\dot{\s}_k^{11}}{\lambda_{k-1}}.
\end{align}
Inserting \eqref{perp}, \eqref{p1}, \eqref{p2}, \eqref{p3} and \eqref{p4} into \eqref{p,l} and combining \eqref{sk-1k-3,p} we obtain 
\begin{align}\label{p,l2}
  &2a\xi_1\gamma_p\f\frac{\ddot{\s}_k^{11,pp}}{\s_k}-\frac{\dot{\s}_k^{pp}}{\s_k}\frac{\dot{\s}_{k-1}^{11}}{\s_{k-1}}-\frac{\dot{\s}_k^{11}}{\s_k}\frac{\dot{\s}_{k-1}^{pp}}{\s_{k-1}}+2\frac{\dot{\s}_{k-1}^{11}\dot{\s}_{k-1}^{pp}}{\s_{k-1}^2}-\frac{\ddot{\s}_{k-1}^{11,pp}}{\s_{k-1}}\r\\  
  \geqslant&-2C\frac{\dot{\s}_k^{11}}{\lambda_{k-1}}|a\gamma_p\xi_1|\geqslant-\e\frac{\dot{\s}_k^{11}\lambda_1^2\gamma_p^2}{\lambda_{k-1}\lambda_k^2}-\frac{Ca^2}{\e}\frac{\dot{\s}_k^{11}\xi_1^2}{\lambda_1^2}
\end{align}
where we use the fact that $\lambda_1>\lambda_{k-1}>\lambda_k$ and $\lambda_k<C(A)$ by Lemma \ref{upperboundforkappak} in the last inequality.
Using \eqref{sl}, we estimate the first term of \eqref{qk}: 
\begin{align}
    -2a^2\xi_1^2\frac{\dot{\s}_{k-1}^{11}}{\s_{k-1}}\f\frac{\dot{\s}_{k-1}^{11}}{\s_{k-1}}-\frac{\dot{\s}_k^{11}}{\s_k}\r=&-2a^2\xi_1^2\frac{\dot{\s}_{k-1}^{11}}{\s_{k-1}}\f\frac{\s_{k-1}-\dot{\s}_k^{11}}{\lambda_1\s_{k-1}}-\frac{\s_k-\s_k(\lambda|1)}{\lambda_1\s_k}\r\nonumber\\
    \geqslant&\frac{2a^2\xi_1^2}{\lambda_1}\f1-\frac{C}{\lambda_1^2}\r\f\frac{\lambda_1\dot{\s}_k^{11}-\s_k}{\lambda_1\s_k}\r
\end{align}
where we use Claim 2, Claim 3 in the last inequality.

Applying Lemma \ref{decom} and inserting \eqref{i,l} and \eqref{p,l2} into \eqref{qk}, we obtain 
\begin{equation}\label{qk2}
\begin{aligned}
      -\frac{\partial^2_{\xi}q_k}{q_k}\geqslant &\frac{2a^2\xi_1^2}{\lambda_1}\f 1-\frac{C}{\lambda_1^2}\r\f\frac{\lambda_1\dot{\s}_k^{11}-\s_k}{\lambda_1\s_k}\r+\sum_{i=2}^{k-1}(C_1-\epsilon)\frac{\dot{\s}_k^{11}\lambda_1^2}{\lambda_i^3}\gamma_i^2\\
      &+\sum_{p=k}^n(C_1-\e)\frac{\dot{\s}_k^{11}\lambda_1^2\gamma_p^2}{\lambda_{k-1}\lambda_k^2}-(n-1)\frac{Ca^2}{\e}\frac{\dot{\s}_k^{11}\xi_1^2}{\lambda_1^2}\\
      \geqslant& \frac{2a^2\xi_1^2}{\lambda_1}\f1-\frac{C}{\lambda_1^2}\r\f\frac{\lambda_1\dot{\s}_k^{11}-\s_k}{\lambda_1\s_k}\r+\sum_{i=2}^{k-1}\frac{C_1}{2}\frac{\dot{\s}_k^{11}\lambda_1^2}{\lambda_i^3}\gamma_i^2+\frac{C_1 }{2}\frac{\dot{\s}_k^{11}\lambda_1^2\gamma_p^2}{\lambda_{k-1}\lambda_k^2}\\
      &-2(n-1)\frac{Ca^2}{C_1}\frac{\dot{\s}_k^{11}\xi_1^2}{ M\lambda_1}
 \end{aligned}   
\end{equation}
by choosing $\e=\frac{C_1}{2}$.
Then combining \eqref{xi2}, \eqref{iq c0} and \eqref{qk2}, we have
\begin{equation}\label{xi3}
    \begin{aligned}
         &-\sum_{p\neq q}\frac{\s_k^{pp,qq}\xi_p\xi_q}{\s_k}+K\frac{(\sum_i\s_k^{ii}\xi_i)^2}{\s_k^2}+2\sum_{i>1}\frac{\s_k^{ii}\xi_i^2}{ (\lambda_1+A+1)\s_k}\\
\geqslant& \frac{2a^2\xi_1^2}{\lambda_1}\f1-\frac{C}{\lambda_1^2}\r\f\frac{\lambda_1\dot{\s}_k^{11}-\s_k}{\lambda_1\s_k}\r+\sum_{i=2}^{k-1}\frac{C_1}{2}\frac{\dot{\s}_k^{11}\lambda_1^2}{\lambda_i^3}\gamma_i^2+\sum_{p=k}^n\frac{C_1 }{2}\frac{\dot{\s}_k^{11}\lambda_1^2\gamma_p^2}{\lambda_{k-1}\lambda_k^2}\\
&-2(n-1)\frac{Ca^2}{C_1}\frac{\dot{\s}_k^{11}\xi_1^2}{ M\lambda_1}+\frac{(k+2)^2}{(k+1)(k+3)}\f1-\frac{C(A)}{\lambda_1}\r^2\frac{\xi_1^2}{\lambda_1^2}\\
&+\f2-\frac{C(k,A)}{M}\r\sum_{l+1\leqslant p\leqslant n}\frac{\s_k^{pp}\xi_p^2}{\lambda_1\s_k}.
    \end{aligned}
\end{equation} 
Under the assumptions in \eqref{sk+1}, we consider
\begin{equation}
 \begin{aligned}
  \sum_{l+1\leqslant p\leqslant n}\frac{\s_k^{pp}\xi_p^2}{\lambda_1\s_k}=&\sum_{l+1\leqslant p\leqslant n}\frac{\s_k^{pp}\f(1+a)\frac{\lambda_p\xi_1}{\lambda_1}+\gamma_p\r^2}{\lambda_1\s_k}\\
\geqslant &(1+a)^2\xi_1^2\sum_{l+1\leqslant p\leqslant n}\frac{\s_k^{pp}\lambda_p^2}{\lambda_1^3\s_k}-2\left|(1+a)\xi_1\sum_{l+1\leqslant p\leqslant n}\frac{\s_k^{pp}\lambda_p\gamma_p}{\lambda_1^2\s_k}\right|\\
\geqslant &(1+a)^2\xi_1^2\sum_{l+1\leqslant p\leqslant n}\frac{\s_k^{pp}\lambda_p^2}{\lambda_1^3\s_k}-2\left|(1+a)\xi_1\sum_{k\leqslant p\leqslant n}\frac{\s_{k-2}(\lambda|p)\lambda_p^2\gamma_p}{\lambda_1^2\s_k}\right|\\
&-2\sum_{l+1\leqslant i\leqslant k-1}\left|(1+a)\xi_1\frac{\s_k^{ii}\lambda_i\gamma_i}{\lambda_1^2\s_k}\right|
\end{aligned}   
\end{equation}
where we use \eqref{sl} and \eqref{perp 3}, which implies that 
\begin{align}
    \sum_{k\leq p\leq n}\dot{\s}_k^{pp}\lambda_p\gamma_p=(\s_{k-1}-\s_{k-2}(\lambda|p)\lambda_p)\lambda_p\gamma_p=-\s_{k-2}(\lambda|p)\lambda_p^2\gamma_p
\end{align} 
in the last inequality.
Then by \eqref{sk+1}, \eqref{sk+1,i}, \eqref{kpsk}, and \eqref{sk,1+}, we have 
\begin{align}
    \sum_{l+1\leqslant p\leqslant n}\frac{\s_k^{pp}\lambda_p^2}{\lambda_1^3\s_k}=&\frac{\sum_{p=l+1}^n\lambda_p\s_k-\sum_{i=1}^{l}\s_{k+1}(\lambda|i)-(k+1-l)\s_{k+1}}{\lambda_1^3\s_k}\nonumber\\
    \geqslant& -(k+1-l)\frac{\s_{k+1}}{\lambda_1^3\s_k}-\frac{C(A)\lambda_1\cdots\lambda_k^2}{M\lambda_1^3\s_k}\nonumber\\
    \geqslant&(k+1-l)(1-\frac{c}{\lambda_1})\frac{\lambda_1\dot{\s}_k^{11}-\s_k}{\lambda_1^2\s_k}-\frac{C(A)}{M}\frac{\dot{\s}_k^{11}}{\lambda_1\s_k}.
\end{align}
Furthermore,
\begin{equation}
 \begin{aligned}
    2\left|(1+a)\xi_1\sum_{k\leqslant p\leqslant n}\frac{\s_{k-2}(\lambda|p)\lambda_p^2\gamma_p}{\lambda_1^2\s_k}\right|\leqslant &2|(1+a)\xi_1\gamma_p|\frac{C\lambda_1\cdots\lambda_{k-2}\lambda_k^2}{\lambda_1^2\s_k}\\
    \leqslant& 2|(1+a)\xi_1\gamma_p|\frac{C\dot{\s}_k^{11}}{\lambda_{k-1}\s_k}\\
    \leqslant&-\frac{C_1}{2}\frac{\dot{\s}_k^{11}\lambda_1^2\gamma_p^2}{\lambda_{k-1}\lambda_k^2}-\frac{2}{C_1}\frac{C'(1+a)^2\dot{\s}_k^{11}\xi_1^2}{ M\lambda_1}
\end{aligned}   
\end{equation}
and for any $l+1\leqslant i\leqslant k-1$
\begin{equation}
 \begin{aligned}
    2\left|(1+a)\xi_1\frac{\s_{k-1}(\lambda|i)\lambda_i\gamma_i}{\lambda_1^2\s_k}\right|=&2\left|(1+a)\xi_1\frac{\f\s_k-\s_k(\lambda|i)\r\gamma_i}{\lambda_1^2\s_k}\right|\\
    \leqslant &2|(1+a)\xi_1\gamma_i|\frac{C\lambda_1\lambda_2\cdots\lambda_{k-1}\lambda_k^2}{\lambda_1^2\lambda_i\s_k}\\
    \leqslant& 2|(1+a)\xi_1\gamma_i|\frac{C\dot{\s}_k^{11}}{\lambda_i}\\
    \leqslant&-\frac{C_1}{2}\frac{\dot{\s}_k^{11}\lambda_1^2\gamma_i^2}{\lambda_i^3}-\frac{2}{C_1}\frac{C'(1+a)^2\dot{\s}_k^{11}\xi_1^2}{M\lambda_1}
\end{aligned}   
\end{equation}
where we use $\lambda_{k-1}\geqslant\lambda_k$ and $\lambda_k<C(A)$ by Lemma \ref{upperboundforkappak}, the assumptions $\k_i\leqslant M$ and $\k_1>M^2$ in Claim 2.
Thus \eqref{xi3} implies that 
\begin{equation}
    \begin{aligned}
         &-\sum_{p\neq q}\frac{\s_k^{pp,qq}\xi_p\xi_q}{\s_k}+K\frac{(\sum_i\s_k^{ii}\xi_i)^2}{\s_k^2}+2\sum_{i>1}\frac{\s_k^{ii}\xi_i^2}{ (\lambda_1+A+1)\s_k}\\
\geqslant& \frac{2a^2\xi_1^2}{\lambda_1}\f1-\frac{C}{ M}\r\f\frac{\lambda_1\dot{\s}_k^{11}-\s_k}{\lambda_1\s_k}\r-2(n-1)\frac{C'a^2}{C_1}\frac{\dot{\s}_k^{11}\xi_1^2}{M\lambda_1}\\
&+\frac{(k+2)^2}{(k+1)(k+3)}\frac{\xi_1^2}{\lambda_1^2}+\f2-\frac{C(k,A)}{M}\r(k+1-l)(1+a)^2\frac{\lambda_1\dot{\s}_k^{11}-\s_k}{\lambda_1\s_k}\\
&-\frac{C(n,k,A)(1+a)^2}{M}\frac{\dot{\s}_k^{11}\xi_1^2}{\lambda_1}\\
    \geqslant &\frac{2(a^2+2(1+a)^2)\xi_1^2}{\lambda_1}\f1-\frac{C_2(k,A)}{M}\r\f\frac{\lambda_1\dot{\s}_k^{11}-\s_k}{\lambda_1\s_k}\r+\frac{(k+2)^2}{(k+1)(k+3)}\frac{\xi_1^2}{\lambda_1^2}\\
&- C_3(n,k,A)(1+a^2)\frac{\dot{\s}_k^{11}\xi_1^2}{M\lambda_1}\\
\geqslant&\frac{8}{5}(a^2+2(1+a)^2)\frac{\dot{\s}_k^{11}\xi_1^2}{\lambda_1\s_k}+\f\frac{(k+2)^2}{(k+1)(k+3)}-\frac{8}{5}(a^2+2(1+a)^2)\r\frac{\xi_1^2}{\lambda_1^2}
    \end{aligned}
\end{equation}
by assuming $M$ large depending on $n$, $k$ and $A$ where $c_0$ obtained from \eqref{sk1 c0}.
Here under the assumption $\s_k(\lambda|1)<-c_0\s_k$, \eqref{sl} implies that
\begin{align}\label{xi4}
  \frac{\dot{\s}_k^{11}}{\s_k}> -\f1+\frac{1}{c_0}\r\frac{\s_k(\lambda|1)}{\lambda_1\s_k}.
\end{align}
Then 
\begin{align}
     -\sum_{p\neq q}\frac{\s_k^{pp,qq}\xi_p\xi_q}{\s_k}+K\frac{(\sum_i\s_k^{ii}\xi_i)^2}{\s_k^2}+2\sum_{i>1}\frac{\s_k^{ii}\xi_i^2}{ (\lambda_1+A+1)\s_k}\geqslant\frac{16}{15}\frac{\dot{\s}_k^{11}\xi_1^2}{\lambda_1\s_k}+\f\frac{(k+2)^2}{(k+1)(k+3)}-\frac{16}{15}\r\frac{\xi_1^2}{\lambda_1^2}
\end{align}

Thus \eqref{xi4} implies that 
\begin{align}
  -\sum_{p\neq q}\frac{\s_k^{pp,qq}\xi_p\xi_q}{\s_k}+K\frac{(\sum_i\s_k^{ii}\xi_i)^2}{\s_k^2}+2\sum_{i>1}\frac{\s_k^{ii}\xi_i^2}{ (\lambda_1+A+1)\s_k} \geqslant \f1+\d_0\r\frac{\dot{\s}_k^{11}\xi_1^2}{\lambda_1\s_k}
\end{align}
where $\d_0=\min\lbrace\frac{1}{15},\frac{1}{(k+1)(k+3)}\rbrace$. 
Hence, we obtain \eqref{main,iq} and complete the proof.
\end{proof}
Due to \eqref{iq c0}, we also improve the results in \cites{Lu2,HZ}.
\begin{lem}\label{lem imp}
    Assume $\lbrace\lambda_i\rbrace\in\Gamma_k$, $\lambda_1\geqslant\lambda_2\cdots\geqslant \lambda_n$, and $\lambda_n>-A$ for some constant $A>0$. There exists a constant $\d_0>0$ and a large constant $C>0$ depending on $n$, $k$, $\s_k$, and $A$, such that if $lambda_1\geqslant C$, then the following inequality holds
    \begin{align}
        -\frac{\s_k^{pp,qq}\xi_p\xi_q}{\s_k}+K\frac{(\sum_i\s_k^{ii}\xi_i)^2}{\s_k^2}+2\sum_{i>1}\frac{\s_k^{ii}\xi_i^2}{ (\lambda_1+A+1)\s_k}\geqslant (1+\d_0)\frac{\xi_1^2}{\lambda_1^2},
    \end{align}
     for some sufficient large $K$ (depending on $\d_0$) where $\xi=( \xi_1,\cdots,\xi_n)$ is an arbitrary vector in $\mathbb{R}^n$.
\end{lem}
\section{Interior estimates of k-Hessian equations with Dirichlet boundary conditions}\label{sec:4}
Somewhere in the sequel, we denote $\s_k$ by $F$ for brevity and define
\begin{align*}
\dot{ F }^{pq}(D^2 u): =\frac{\partial F}{\partial u_{pq} },\quad
\ddot{F}^{pq,rs}(D^2 u): =\frac{\partial^2 F}{\partial u_{qp}\partial u_{rs}}.
\end{align*}
If We further assume that $\lbrace u_{ij} \rbrace$ is diagonal at a certain point, then we define
\begin{align*}
\frac{\partial F}{\partial\lambda_p}(\lambda(D^2 u))=\dot{F}^{pp}(D^2 u),\quad\frac{\partial^2 F}{\partial\lambda_p\partial\lambda_q}(\lambda(D^2 u))=\ddot{F}^{pp,qq}(D^2 u).
\end{align*}

We first proof Theorem \ref{thm sc}.
\begin{proof}
we consider the function for any $\e>0$, $x \in \Omega$, $\xi \in \mathbb{S}^{n-1}$
\begin{eqnarray*}
\widetilde{P}(x, \xi)=\b\log (-u)+\log \max \{u_{\xi\xi}(x), 1\}+\frac{L}{2}|Du|^2,
\end{eqnarray*}
where $L\gg 1$, $\b>0$ are constants to be determined later. Clearly, the maximum of $\widetilde{P}$ is attained
at some interior point $x_0$ of $\Omega$ and some $\xi(x_0) \in \mathbb{S}^{n-1}$.
Choose canonical coordinate frames $e_1, \ldots, e_n$ at $x_0$ such that $\xi(x_0)=e_1$ and $\{u_{ij} (x_0)\}$ is diagonal. Let the eigenvalues $\lambda$ of $D^2 u$ be ordered as
$$\lambda_1(x_0)\geqslant\lambda_2(x_0)\geqslant...\geqslant \lambda_n(x_0).$$ We construct a unit vector field $v(s)$ near $p$ such that $v(0)=e_1(x_0)$ and $v'(0)=\sum_{p\neq 1}\frac{u_{1pi}}{\lambda_1+1-\lambda_p}e_p$.
We may also assume that $\lambda_1(x_0)> 1$ is sufficiently large.
Then we consider the function
\begin{eqnarray*}
P(x)=\b\log (-u)+\log\f \nabla^2 u\f v(s),v(s)\r+\f g(v(s),e_1(s))\r^2-1\r+\frac{L}{2}|Du|^2.
\end{eqnarray*}
Note that $x_0$ is also a maximum point of $P$. At $x_0$, we have 
\begin{align}\label{cri}
    0=\nabla_iP(x_0)=\b\frac{u_i}{u}+\frac{u_{11i}}{u_{11}}+Lu_i\lambda_i
\end{align}
   and 
   \begin{equation}\label{hess 31}
   \begin{aligned}
    0\geqslant &  \nabla_{ij}P(x_0) \\
    =&\b\frac{u_{ij}}{u}-\b\frac{u_i u_j}{u^2}+\frac{u_{11ij}}{u_{11}}+\sum_{l>1}\frac{u_{1li}^2}{u_{11}(u_{11}-u_{ll}+1)}\\
    &-\frac{u_{11i}u_{11j}}{u_{11}^2}+L u_{ki}u_{kj}+L u_{kij}u_k.
   \end{aligned}    
   \end{equation}
By $\s_k(D^2 u)=f(x,u,\nabla u)$, we have at $x_0$ 
\begin{equation}\label{gra}
   \nabla_l\s_k=f_l+f_uu_l+f_{u_l}u_{ll}
\end{equation}
and 
\begin{equation}\label{hess s}
\nabla^2_{11}\s_k=\dot{F}^{ij}u_{ij11}+\ddot{F}^{pq,rs}u_{pq1}u_{rs1}=f_{11}.
\end{equation}
Besides, there exists some constant $C>0$ depending on $\|f\|_{C^2}$ and $\|u\|_{C^1}$ such that 
\begin{align}
    |f_1|\leqslant C+Cu_{11},\quad f_{11}\geqslant-C-Cu_{11}^2+f_{u_l}u_{11l}.
\end{align}
Multiplying \eqref{hess 31} by using $\dot{F}^{ij}$, by \eqref{cri}, \eqref{gra} and \eqref{hess s} we obtain 
\begin{equation}\label{hess 32}
\begin{aligned}
    0\geqslant& \b\frac{kF}{u}-\b\frac{\dot{F}^{ii}u_i^2}{u^2}+\frac{\dot{F}^{ii}u_{11ii}}{u_{11}}+2\sum_{p>1}\frac{\dot{F}^{ii}u_{1pi}}{u_{11}(u_{11}-\lambda_p+1)}-\frac{\dot{F}^{ii}u_{11i}^2}{u_{11}^2}+L\dot{F}^{ii}u_{ii}^2+L\nabla_l\s_k u_l\\
    \geqslant &\b\frac{kF}{u}-\b\frac{\dot{F}^{ii}u_i^2}{u^2}-\frac{\dot{F}^{pq,rs}u_{pq1}u_{rs1}}{u_{11}}+\frac{f_{u_l}u_{11l}}{u_{11}}+2\sum_{p>1}\frac{\dot{F}^{ii}u_{1pi}}{u_{11}(u_{11}-u_{pp}+1)}\\
    &-\frac{\dot{F}^{ii}u_{11i}^2}{u_{11}^2}+L\dot{F}^{ii}u_{ii}^2-Cu_{11}+Lf_{u_l}\lambda_l u_l-CL\\
    \geqslant &\b\frac{kF}{u}-\b\frac{\dot{F}^{ii}u_i^2}{u^2}-\frac{\dot{F}^{pq,rs}u_{pq1}u_{rs1}}{u_{11}}+2\sum_{p>1}\frac{\dot{F}^{ii}u_{1pi}}{u_{11}(u_{11}-u_{pp}+1)}-\frac{\dot{F}^{ii}u_{11i}^2}{u_{11}^2}+L\dot{F}^{ii}u_{ii}^2\\
    &-\b\frac{f_{u_l}u_l}{u}-Cu_{11}-CL.
\end{aligned}
\end{equation}
Referring to Andrews \cite{And07}, we have 
\begin{equation}\label{fppqq1}
 \begin{aligned}
-\ddot{F}^{pq,rs}u_{pq1}u_{rs1}=&-\ddot{F}^{pp,qq}u_{pp1}u_{qq1}+2\sum_{p>q}\ddot{F}^{pp,qq}u_{pq1}^2\\
\geqslant& -\ddot{F}^{pp,qq}u_{pp1}u_{qq1}+2\sum_{\lambda_p<\lambda_1}\frac{\dot{F}^{pp}-\dot{F}^{11}}{\lambda_1-\lambda_p}u_{11p}^2.
\end{aligned}   
\end{equation}
Besides, 
\begin{eqnarray}\label{fppqq2}
 \sum_{1<p,1\leqslant i\leqslant n}\frac{2\dot{F}^{ii}u_{1pi}^2}{(u_{11}+1- u_{pp})}\geqslant  2\sum_{p>1}\frac{\dot{F}^{11}u_{11p}^2}{(u_{11}-u_{pp}+1)}+ 2\sum_{i>1}\frac{\dot{F}^{ii}u_{1ii}^2}{(u_{11}-u_{ii}+1)}.
\end{eqnarray}
Using the critical equation \eqref{cri}, for any $i\geqslant 2$ we have 
\begin{align}
   -\b\frac{\dot{F}^{ii}u_i^2}{u^2}=-\frac{1}{\b}\f\frac{u_{11i}}{u_{11}}+Lu_iu_{ii}\r^2\geqslant -\frac{2}{\b}\frac{\dot{F}^{ii}u_{11i}^2}{u_{11}^2}-2\frac{L^2}{\b}\dot{F}^{ii}u_i^2u_{ii}^2.
\end{align}
Then we choose $\b=L^2\gg 4$ and obtain 
\begin{equation}\label{fppqq3}
    \begin{aligned}
&-\frac{\dot{F}^{pq,rs}u_{pq1}u_{rs1}}{u_{11}}+2\sum_{p>1}\frac{\dot{F}^{ii}u_{1pi}^2}{u_{11}(u_{11}-u_{pp}+1)}-\f 1+\frac{2}{\b}\r\sum_{i\geqslant2}\frac{\dot{F}^{ii}u_{11i}^2}{u_{11}^2}-\frac{\dot{F}^{11}u_{111}^2}{u_{11}^2}\\
\geqslant&-\frac{\ddot{F}^{pp,qq}u_{pp1}u_{qq1}}{u_{11}}+2\sum_{p>1}\frac{\dot{F}^{pp}}{u_{11}({u_{11}-u_{pp}+1})}u_{11p}^2+ 2\sum_{i>1}\frac{\dot{F}^{ii}u_{1ii}^2}{u_{11}(u_{11}-u_{ii}+1)}\\
&-\frac{3}{2}\sum_{i\geqslant2}\frac{\dot{F}^{ii}u_{11i}^2}{u_{11}^2}-\frac{\dot{F}^{11}u_{111}^2}{u_{11}^2}\\
=&-\frac{\ddot{F}^{pp,qq}u_{pp1}u_{qq1}}{u_{11}}+ 2\sum_{i>1}\frac{\dot{F}^{ii}u_{1ii}^2}{u_{11}(u_{11}+A+1)}-\frac{\dot{F}^{11}u_{111}^2}{u_{11}^2}+\frac{\dot{F}^{ii}u_{11i}^2\f\frac{1}{2}u_{11}-\frac{3}{2}A-\frac{3}{2}\r}{u_{11}^2(u_{11}-u_{ii}+1)}\\
\geqslant &-\frac{\ddot{F}^{pp,qq}u_{pp1}u_{qq1}}{u_{11}}+ 2\sum_{i>1}\frac{\dot{F}^{ii}u_{1ii}^2}{u_{11}(u_{11}+A+1)}-\frac{\dot{F}^{11}u_{111}^2}{u_{11}^2}
\end{aligned}
\end{equation}
by assuming $\lambda_1\geqslant 3A+3$ in the last inequality.

Due to Lemma \ref{iq sc}, combining \eqref{fppqq1} and \eqref{fppqq2}, \eqref{hess 32} becomes 
\begin{equation}\label{hess 33}
\begin{aligned}
    0\geqslant &-L^2\frac{\dot{F}^{11}u_1^2}{u^2}-2\dot{F}^{ii}u_i^2\lambda_i^2-\frac{\dot{F}^{pp,qq}u_{pp1}u_{qq1}}{u_{11}}+K\frac{(\nabla_1\s_k)^2}{u_{11}\s_k}+2\sum_{i>1}\frac{\dot{F}^{ii}u_{1ii}^2}{u_{11}(u_{11}+A+1)}-\frac{\dot{F}^{11}u_{111}^2}{u_{11}^2}\\
    &+L\dot{F}^{ii}u_{ii}^2-\frac{C_3L^2}{u}-C_4L-C_5(K)u_{11}\\
   \geqslant &\f\frac{Lu_{11}^2}{5}-\frac{C_1L^2}{u^2}\r\dot{F}^{11}+\f\frac{L}{5}-C_2\r\sum_{i=1}^n\dot{F}^{ii}u_{ii}^2+L\f\frac{k\s_k}{5n}u_{11}-C_3\frac{L}{u}\r\\
   &+L\f\frac{k\s_k}{5n}u_{11}-C_4\r+\f\frac{Lk\s_k}{5n}-C_5(K)\r u_{11}.
\end{aligned}
\end{equation}
where we use property \eqref{fk1} in the last inequality. In \eqref{hess 33}, $C_i$, $ i=1,2,\dots,5$ all depend on $n$, $k$, $\|f\|_{C^2}$ and $\|u\|_{C^1}$. We choose $L=\max\lbrace 4,5C_2,\frac{5nC_5}{kf}\rbrace$ so that the second term and the last term in the last inequality of \eqref{hess 33} are positive. Then we obtain $\max\lambda_1 u^{\b}\leqslant \max\lbrace\sqrt{5C_1L}\sup u^{\b-1},\frac{5nLC_3\sup u^{\b-1}}{kf_0}, \frac{5nC_4\sup u^{\b}}{kf_0}\rbrace$ and complete the proof.
\end{proof}
\section{A Liouville problem for general k-Hessian equation}\label{sec:5}
We first consider the interior $C^2$ estimates for the following Dirichlet problem:
\begin{equation}\label{eq rig}
\left\{
\begin{aligned}
&\sigma_k(D^2 u)=1 \quad x \in \Omega, \\
&u=0 \quad x \in \partial\Omega.    
\end{aligned}
\right.
\end{equation}
\begin{lem}\label{est, liv}
Assume that $D^2 u\geqslant -AI$ for some positive constant $A>0$.
Then, for any $k$-convex solution $u \in C^2(\Omega)\cap C^{0,1}(\overline{\Omega})$, there exists some $\a>1$ and $C>0$ depending only on $n$, $k$, and the diameter of $\Omega$, such that
\begin{eqnarray}\label{est}
\max_{x \in \Omega}(-u(x))^{\alpha}|D^2 u(x)|\leqslant C.
\end{eqnarray}
\end{lem}

\begin{proof}
Using the comparison principle (see \cites{CNS, GT} for details), there exist large $a$ and $b$ depending on the diameter of $\Omega$, such that
$$w=a|x|^2-b$$
and
$$w\leqslant u\leqslant 0.$$
Thus, $-u$ can be controlled by the diameter of the domain $\Omega$.

Now we consider the function for $x \in \Omega$, $\xi \in \mathbb{S}^{n-1}$
\begin{eqnarray*}
\widetilde{P}(x, \xi)=\alpha\log (-u)+\log \max \{u_{\xi\xi}(x), 1\}+\frac{1}{2}|x|^2,
\end{eqnarray*}
where $\alpha>k-1$ is a constant to be determined later.
Clearly, the maximum of $\widetilde{P}$ is attained
at some interior point $x_0$ of $\Omega$ and some $\xi(x_0) \in \mathbb{S}^{n-1}$.
Choose canonical coordinate frames $e_1, \ldots, e_n$ at $x_0$ such that $\xi(x_0)=e_1$ and $\{u_{ij} (x_0)\}$ is diagonal. Let the eigenvalues $\lambda$ of $D^2 u$ be ordered as
$$\lambda_1(x_0)\geqslant\lambda_2(x_0)\geqslant...\geqslant \lambda_n(x_0).$$ We construct a unit vector field $v(s)$ near $p$ such that $v(0)=e_1(x_0)$ and $v'(0)=\sum_{p\neq 1}\frac{u_{1pi}}{\lambda_1+1-\lambda_p}e_p$.
We may also assume that $\lambda_1(x_0)> 1$ is sufficiently large.
Then we consider the function
\begin{eqnarray*}
P(x)=\a\log (-u)+\log\f \nabla^2 u\f v(s),v(s)\r+\f g(v(s),e_1(s))\r^2-1\r+\frac{1}{2}|x|^2.
\end{eqnarray*}
Note that $x_0$ is also a maximum point of $P$.
At the maximum point $x_0$,
\begin{eqnarray}\label{cri rig}
0=P_i=\frac{\alpha u_i}{u}+\frac{u_{11i}}{u_{11}}+ x_i,
\end{eqnarray}
and
\begin{eqnarray}\label{hess p}
P_{ij}=\frac{\alpha u_{ij}}{u}-\frac{\alpha u_i u_j}{u^2}
+\frac{u_{11ij}}{u_{11}}+\sum_{p>1}\frac{2u_{1pi}^2}{ u_{11}( u_{11}+1- u_{pp})}-\frac{u_{11i}u_{11j}}{u_{11}^2}+\delta_{ij}.
\end{eqnarray}
Multiplying \eqref{hess p} by $\dot{F}^{ij}$, at $x_0$ we have 
\begin{eqnarray}\label{hess p1}
0\geqslant \a\frac{kF}{u}-\a\frac{\dot{F}^{ii}u_i^2}{u^2}
+\frac{\dot{F}^{ii}u_{11ii}}{u_{11}}+2\sum_{p>1}\frac{\dot{F}^{ii}u_{1pi}^2}{ u_{11}( u_{11}+1- u_{pp})}-\frac{\dot{F}^{ii}u_{11i}^2}{u_{11}^2}+(n-k+1)\s_{k-1}.
\end{eqnarray}
Similar to the proof in Theorem \ref{thm sc}, using the critical equation \eqref{cri rig} we obtain at $x_0$ 
\begin{align}
   -\a\frac{\dot{F}^{ii}u_i^2}{u^2}=-\frac{1}{\a}\dot{F}^{ii}\f\frac{u_{11i}}{u_{11}}+x_i\r^2\geqslant -\frac{2}{\a}\frac{\dot{F}^{ii}u_{11i}^2}{u_{11}^2}-\frac{2}{\a}\dot{F}^{ii}x_i^2.
\end{align}
Taking the derivatives of the equation \eqref{eq rig}, we obtain
$$\dot{F}^{ij}u_{ijk}=0,$$
and
\begin{eqnarray*}
\ddot{F}^{ij, kl}u_{kl1}u_{ij1}+\dot{F}^{ij}u_{ij11}=0.
\end{eqnarray*}
By \eqref{fppqq1} and \eqref{fppqq2}, similar to \eqref{fppqq3} we have 
by choosing $\a\gg 4$ 
\begin{equation}
    \begin{aligned}
&-\frac{\dot{F}^{pq,rs}u_{pq1}u_{rs1}}{u_{11}}+2\sum_{p>1}\frac{\dot{F}^{ii}u_{1pi}^2}{u_{11}(u_{11}-u_{pp}+1)}-\f 1+\frac{2}{\a}\r\sum_{i\geqslant2}\frac{\dot{F}^{ii}u_{11i}^2}{u_{11}^2}-\frac{\dot{F}^{11}u_{111}^2}{u_{11}^2}\\
\geqslant &-\ddot{F}^{pp,qq}u_{pp1}u_{qq1}+ 2\sum_{i>1}\frac{\dot{F}^{ii}u_{1ii}^2}{u_{11}(u_{11}+A+1)}-\frac{\dot{F}^{11}u_{111}^2}{u_{11}^2}
\end{aligned}
\end{equation}
by assuming $\lambda_1\geqslant 3A+3$ in the last inequality.
Then applying Lemma \ref{iq sc}, we obtain
\begin{equation}\label{hess p2}
   \begin{aligned}
0\geqslant &\a\frac{kF}{u}-\frac{2}{\a}\dot{F}^{ii}X_i^2-\ddot{F}^{pp,qq}u_{pp1}u_{qq1}+ 2\sum_{i>1}\frac{\dot{F}^{ii}u_{1ii}^2}{u_{11}(u_{11}+A+1)}\\
&-\f1+\frac{2}{\a}\r\frac{\dot{F}^{11}u_{111}^2}{u_{11}^2}+(n-k+1)\s_{k-1}\\
\geqslant &\a\frac{kF}{u}+\f\d_0-\frac{2}{\a}\r\frac{\dot{F}^{11}u_{111}^2}{u_{11}^2}+\f n-k+1-\frac{2\sup_{\Omega}|X|^2}{\a}\r\s_{k-1}\\
\geqslant &\frac{C_1\a}{u}+\frac{n-k+1}{2}\s_{k-1}
\end{aligned} 
\end{equation}
by choosing $\a=\max\lbrace 4, k-1, \frac{2}{\d_0}, \frac{4\sup_{\Omega}|X|^2}{n-k+1}\rbrace$.
Thus, we employ Lemma \ref{sk-1} to obtain that there exists some $C>0$ only depending on $n$ and the diameter of $\Omega$, such that 
\begin{align}
   C(-u)u_{11}^{\frac{1}{k-1}} \leqslant (-u)\s_{k-1}\leqslant \frac{2C_1\a}{n-k+1}.
\end{align}
Then we complete the proof.
\end{proof}

We now begin to prove Theorem \ref{thm rig}.
\begin{proof}
The proof is standard (refer to \cites{LRW,TW} for more details). Let $u$ be an entire solution of \eqref{sk, rig}. We consider the domain
$$\Omega_{R}=\{y \in \mathbb{R}^n: u(Ry)<R^2\}$$ for any large $R>1$.
Let
$$u_{R}(y)=\frac{u(Ry)-R^2}{R^2}.$$
Then 
\begin{align}\label{hess ur}
 D^2_{y} u_R=D^2_{x} u.  
\end{align}
Hence $u_R$ satisfying the following Dirichlet problem:
\begin{equation}\label{SQ-1}
\left\{
\begin{aligned}
&\sigma_k(D^2u_R(x))=1 \quad x \in \Omega_R, \\
&u_R=0 \quad x \in \partial\Omega_R.
\end{aligned}
\right.
\end{equation}
The quadratic growth condition in Theorem \ref{thm rig} implies that
\begin{eqnarray*}
a|R y|^2-b\leqslant u(R y)\leqslant R^2.
\end{eqnarray*}
Then
\begin{eqnarray*}
|y|^2\leqslant \frac{1+b}{a},
\end{eqnarray*}
which means that $\Omega_R$ is bounded. 
Applying Lemma \ref{est, liv}, we obtain
\begin{equation}\label{C rig}
(-u_R)^{\alpha}|D^2 u_R|\leqslant C
\end{equation}
where $C$ is independent of $R$.
Now we
consider the domain
$$\Omega^{\prime}_{R}=\{y \in \mathbb{R}^n: u(Ry)<\frac{R^2}{2}\}\subset \Omega_R.$$
In $\Omega^{\prime}_{R}$, we have
$$u_R(y)\leqslant-\frac{1}{2}.$$
Due to \eqref{hess ur} and \eqref{C rig}, 
$$D^2_{x} u=|D^2 u_R|\leqslant 2^{\alpha} C.$$
The estimate above holds true for all $R$. Applying Evans-Krylov theory we have
$$|D^2 u|_{C^{\alpha}(B_R)}\leqslant C(n, \alpha)\frac{|D^2 u|_{C^{0}(B_{2R})}}{R^{\alpha}}\leqslant \frac{C(n, \alpha)}{R^{\alpha}}.$$
Therefore, we obtain our theorem by letting $R\rightarrow +\infty$.
\end{proof}

\section{Global curvature estimate for convex hypersurfaces}\label{sec:6}
In this section, we focus on semiconvex hypersurfaces with the prescribed $k$-th mean curvature and complete the proof of Theorem \ref{thm wsc} and Theorem \ref{thm 2d}. 

\emph{Proof of Theorem \ref{thm  wsc}}
We consider the test function
\[Q=\log \k_1-B\log u\]
where $B$ is a large constant to be considered. Furthermore, we assume that $Q$ attains its maximum at the point $X_0$. Similarly to the proof in \cite{Chu}, also in the proof in Theorem \ref{thm sc} and Theorem \ref{thm rig}, we consider the perturbed quantity
\[\tilde{Q}=\log\tilde{\k}_1-Bu,\]
which also attains its maximum at point $X_0$. 

At $X_0$, we have 
\begin{align}
    0=\frac{h_{11i}}{\k_1}-Bh_{ii}\metric{X}{\partial_i}
\end{align}
and due to Lemma 3.1 in \cite{Chu}
\begin{equation}\label{hess c}
    \begin{aligned}
        0\geqslant& 2\sum_{p>1}\frac{\dot{F}^{11,pp}h_{11p}^2}{\k_1^2}+2\sum_{i>1}\frac{\dot{F}^{11}h_{11i}^2}{\k_1(\k_1-\k_i+1)}-\frac{\ddot{F}^{pp,qq}h_{pp1}h_{qq1}}{\k_1}+2\sum_{p>1}\frac{\dot{F}^{pp}h_{1pp}^2}{\k_1(\k_1-\k_p+1)}\\
        &-\sum_{p=1}^n\frac{\dot{F}^{pp}h_{11p}^2}{\k_1^2}
        +\sum_{i=1}^n\f\frac{B}{C}-C\r\dot{F}^{ii}\k_i^2-CB.
    \end{aligned}
\end{equation}
Then
\begin{equation}\label{fppqq c}
    \begin{aligned}
&2\sum_{p>1}\frac{\dot{F}^{11,pp}h_{11p}^2}{\k_1^2}+2\sum_{i>1}\frac{\dot{F}^{11}h_{11i}^2}{\k_1(\k_1-\k_i+1)}-\sum_{p=1}^n\frac{\dot{F}^{pp}h_{11p}^2}{\k_1^2}\\
\geqslant &2\sum_{i>1}\frac{\dot{F}^{ii}h_{11i}^2}{\k_1(\k_1-\k_i+1)}-\sum_{p=1}^n\frac{\dot{F}^{pp}h_{11p}^2}{\k_1^2}\\
\geqslant &\sum_{i>1}\frac{\k_1+\k_i-1}{\k_1-\k_i+1}\frac{\dot{F}^{ii}h_{11i}^2}{\k_1}-\frac{\dot{F}^{11}h_{111}^2}{\k_1^2}.
    \end{aligned}
\end{equation}
According to Lemma 11 in \cite{RW2}, we have 
\[-\k_i\leqslant \frac{n-k}{k}\k_1\]
which implies that
\begin{align}\label{2k>n}
 \k_1+\k_i-1=\frac{2k-n}{k}\k_1-1>0
\end{align}
by assuming that $2k>n$ and $\k_1\gg 1$.
Or under the assumption of Theorem \ref{thm  wsc} that $\k_i\geqslant -A$, we have
\begin{align}\label{sc}
 \k_1+\k_i-1\geqslant \k_1-A-1>0   
\end{align}
by assuming $\k_1\geqslant A+1$ rather than the condition $2k>n$.
Plugging \eqref{fppqq c} and \eqref{sc} into \eqref{hess c}, we have 
\begin{equation}\label{hess c2}
    \begin{aligned}
        0\geqslant& -\frac{\ddot{F}^{pp,qq}h_{pp1}h_{qq1}}{\k_1}+2\sum_{p>1}\frac{\dot{F}^{pp}h_{1pp}^2}{\k_1(\k_1-\k_p+1)}-\frac{\dot{F}^{11}h_{111}^2}{\k_1^2}
        +\sum_{i=1}^n\f\frac{B}{C}-C\r\dot{F}^{ii}\k_i^2-CB\\
        \geqslant&\sum_{i=1}^n\f\frac{B}{C}-C\r\dot{F}^{ii}\k_i^2-K\frac{(\dot{F}^{ii}h_{ii1})^2}{\s_k\k_1}-CB.
    \end{aligned}
\end{equation}
Here we use the crucial lemma \ref{iq sc}.
Differentiating \eqref{eq c}, we have
\begin{align}
    \sum_i\dot{F}^{ii}h_{ii1}=h_{11}d_{\nu} f(e_1)+d_xf(e_1)
\end{align} 
which implies that
\begin{align}
   \frac{\sum_i(\dot{F}^{ii}h_{ii1})^2}{\s_k\k_1}\leqslant C\k_1. 
\end{align}
Then combining Property \eqref{fk1}, we have 
\begin{equation}
    \begin{aligned}
        0\geqslant&\sum_{i=1}^n\f\frac{kB}{nC}-\frac{k}{n}C-CK\r\k_1-CB,
    \end{aligned}
\end{equation}
which implies $\k_1$ is bounded above by some constant $C>0$ only depending on $n$, $k$, $\|M\|_{C^1}$, $\inf f$ and $\|f\|_{C^2}$.\qed
\begin{rem}
    If we further assume $2k>n$, following the proof in Theorem \ref{thm  wsc}, we can derive the global $C^2$ estimates of the $k$-convex hypersurface satisfying \eqref{eq c} without the assumption $\k_i>-A$ under the following conjecture. 
    
\textbf{Conjecture 1.5}
   Suppose that $2k>n$, and $\k=(\k_1,\dots,\k_n)\in \Gamma_k$. If we assume that $\k_1\geqslant\cdots\geqslant\k_n$ and $\k_1$ is large enough, then for an arbitrary $n$-vector $\xi$ in $\R^n$ and some sufficient large $K>0$, the following inequality holds
    \begin{align}
       \k_1\f K(\dot{F}^{jj}\xi_j)^2-\ddot{\s}_k^{pp,qq}\xi_p\xi_q\r-\dot{F}^{11}\xi_1^2+(\dot{F}^{jj}+(\k_1+\k_j)\ddot{\s}_k^{11,jj})\geqslant 0.  
    \end{align}
One can refer to \cite{Tu} for a detailed proof. We note that this conjecture is a weakened version of Conjecture 2 by
in \cite{RW2} but exhibits the same ability to derive global $C^2$ estimates of the curvature equation \eqref{eq c}. Our crucial lemma \ref{iq sc} is a semi-convexity version of Congjecture 1.5.
\end{rem}

Finally, we provide a concise proof of Theorem \ref{thm 2d} using the crucial Lemma \ref{iq sc} when $k=2$.

\emph{Sketch proof of Theorem \ref{thm 2d}}
We consider the test function for some $N\gg 1$ and small $\e>0$
\begin{align}
    Q=\log \k_1-\e u+\frac{N}{2}|x|^2
\end{align}
in this case. We assume that $Q$ attains its maximum at point $X_0$. Similar to the proof in Theorem \ref{thm  wsc} using the idea in \cite{Chu}, we disturb $Q$ as below 
\begin{align}
    \tilde{Q}=\log \tilde{\k}_1-\e u+\frac{N}{2}|x|^2
\end{align}
which also attains its maximum at $X_0$.
At $X_0$, we have 
\begin{align}\label{cri 2d}
    0=\frac{h_{11i}}{\k_1}-\e \frac{h_{ii}\metric{X}{\partial_i}}{u}+Nx_i
\end{align}
and 
\begin{equation}\label{hess 2d}
    \begin{aligned}
      0\geqslant& 2\sum_{p>1}\frac{\dot{F}^{11,pp}h_{11p}^2}{\k_1^2}+2\sum_{i>1}\frac{\dot{F}^{11}h_{11i}^2}{\k_1(\k_1-\k_i+1)}-\frac{\ddot{F}^{pp,qq}h_{pp1}h_{qq1}}{\k_1}+2\sum_{p>1}\frac{\dot{F}^{pp}h_{1pp}^2}{\k_1(\k_1-\k_p+1)}\\
        &-\sum_{p=1}^n\frac{\dot{F}^{pp}h_{11p}^2}{\k_1^2}+\sum_{i=1}^n\e\dot{F}^{ii}\k_i^2+N\sum_i\dot{F}^{ii}-2NuF-C\k_1-C\\
        \geqslant &-\frac{\ddot{F}^{pp,qq}h_{pp1}h_{qq1}}{\k_1}+2\sum_{p>1}\frac{\dot{F}^{pp}h_{1pp}^2}{\k_1(\k_1-\k_p+1)}+2\sum_{i>1}\frac{\dot{F}^{ii}h_{11i}^2}{\k_1(\k_1-\k_i+1)}-\sum_{p=1}^n\frac{\dot{F}^{pp}h_{11p}^2}{\k_1^2}\\
        &+\sum_{i=1}^n\e\dot{F}^{ii}\k_i^2+N(n-1)\s_1-C\k_1-CN.
    \end{aligned}
\end{equation}
If we further assume that $-\k_i<A$ for any $1<i\leqslant n$, then applying Lemma \ref{iq sc} similar to the proof in \eqref{hess c2} we have 
\begin{equation}\label{hess 2d2}
    \begin{aligned}
      0 \geqslant &-\frac{\ddot{F}^{pp,qq}h_{pp1}h_{qq1}}{\k_1}+2\sum_{p>1}\frac{\dot{F}^{pp}h_{1pp}^2}{\k_1(\k_1-\k_p+1)}+K\frac{(\dot{F}^{ii}h_{ii1})^2}{\s_k\k_1}-\frac{\dot{F}^{11}h_{111}^2}{\k_1^2}\\
        &+\sum_{i=1}^n\e\dot{F}^{ii}\k_i^2+N(n-1)\s_1-(C+CK)\k_1-CN\\
        \geqslant &\f N(n-1)-C-CK\r\k_1-CN.
    \end{aligned}
\end{equation}

By choosing $N$ large enough (independent of $A$), we find that $\k_1$ is bounded above by some constant $C$. Both $N$ and $C$ depend only on $n$, $k$, $\|f\|_{C^2}$ $\inf f$ and $\|M\|_{C^1}$.
If $|\k_n|\geqslant A$, then \eqref{hess 2d} becomes
\begin{equation}\label{hess 2d3}
    \begin{aligned}
        0\geqslant &-\frac{\ddot{F}^{pp,qq}h_{pp1}h_{qq1}}{\k_1}+\frac{(\dot{F}^{ii}h_{ii1})^2}{\s_k\k_1}-\sum_{p=1}^n\frac{\dot{F}^{pp}h_{11p}^2}{\k_1^2}+\sum_{i=1}^n\e\dot{F}^{ii}\k_i^2+N(n-1)\s_1-C\k_1-CN\\
        \geqslant &-2C\e^2\sum_{p=1}^n\dot{F}^{pp}\k_p^2-2N^2\sum_{p=1}^n\dot{F}^{pp}X_p^2+\sum_{i=1}^n\e\dot{F}^{ii}\k_i^2+N(n-1)\s_1-C\k_1-CN\\
        \geqslant& \f-2C\e^2+\frac{\e}{2}\r\sum_{p=1}^n\dot{F}^{pp}\k_p^2+(-2nN^2|X|^2+\frac{\e}{2}\k_n^2)\dot{F}^{nn}+N(n-1)\s_1-C\k_1-CN.
    \end{aligned}
\end{equation}
where we employ \eqref{logsk} in the first inequality and the critical equation \eqref{cri 2d} in the second inequality. By setting $\e=\frac{1}{4c}$ and $A=4N\sqrt{nC}\sup|X|$, the proof can be completed in a way similar to the one presented above.

\end{document}